\documentclass[BCOR8mm,DIV14]{scrartcl}
\usepackage{bbm}
\usepackage[english]{babel}
\usepackage[all]{xy}
\usepackage{amsmath,amssymb,amsthm}
\newtheorem{theorem*}{Theorem}
\newtheorem{theorem}{Theorem}
\newtheorem{conjecture}{Conjecture}
\newtheorem{corollary}{Corollary}
\newtheorem{proposition}{Proposition}      
\newtheorem{lemma}{Lemma}
\newtheorem{remark}{Remark}
\newtheorem{definition}{Definition}  
\def\leq{\leqslant}
\def\geq{\geqslant}       
\begin{document}
\title{  Angles and a Classification of Normed Spaces   }   
\author{   VOLKER W. TH\"UREY      
               \\   Rheinstr. 91  \\  28199 Bremen, \ Germany       
               \thanks{volker@thuerey.de \ \ T: 49 (0)421/591777 }   }
\maketitle
\centerline{ 2010  AMS-classification:  \ 46B20, 52A10  }    
\centerline{ Keywords: \ generalized angle, normed space  } 
   \begin{abstract}
        \centerline{Abstract}
         We suggest a concept of generalized `angles' in arbitrary real normed vector spaces. 
         We give for each real number a definition of an  `angle'  by means of the shape of the unit ball. 
         They all  yield  the  well known Euclidean angle in the special case of real inner product  spaces. 
         With these different 
         angles we achieve  a classification of normed spaces, and we obtain a characterization of inner product
         spaces. Finally we consider  this construction  also  for a  generalization of normed spaces,  i.e.
         for spaces   which may have a non-convex   unit ball.     
   \end{abstract} 
	 \tableofcontents 
	 \newpage   	
	  	    \section{Introduction}      \ \ \  \ \ 
	   In a real  inner product space $ \left( X , <  .  | . > \right) $ it is well-known that 
	   the inner product can be expressed by the norm, namely for $ \vec{x} , \vec{y} \in X $,   
	    $ \vec{x} \neq \vec{0} \neq \vec{y} $, we can write
	  \begin{equation*}  
           <\vec{x} \ | \ \vec{y}> \  =  \ 
          \frac{1}{4} \cdot \left( \: \|\vec{x}+\vec{y}\|^{2} -   \|\vec{x}-\vec{y}\|^{2} \: \right)   
          \  =  \  \frac{1}{4} \cdot  \|\vec{x}\| \cdot \|\vec{y}\| \cdot 
          \left[  \ \left\| \frac{\vec{x}}{\|\vec{x}\|}   +  \frac{\vec{y}}{\|\vec{y}\|} \right\|^{2} \  - \
          \left\| \frac{\vec{x}}{\|\vec{x}\|}   -   \frac{\vec{y}}{\|\vec{y}\|} \right\|^{2} \  \right] \ .     
    \end{equation*} 
    Furthermore we have  for all  $ \vec{x},\vec{y} \neq \vec{0} $ the  usual Euclidean angle  
   \begin{align*} 
         \angle_{Euclid} (\vec{x}, \vec{y}) \  := 
         \arccos{\frac{< \vec{x} \: | \: \vec{y} >}{ \|\vec{x}\| \cdot \|\vec{y}\|}}   
         \  = \  \arccos{    \left( \ \frac{1}{4} \cdot   
         \left[ \ \left\| \frac{\vec{x}}{\|\vec{x}\|} + \frac{\vec{y}}{\|\vec{y}\|} \right\|^{2} \ - \
         \left\| \frac{\vec{x}}{\|\vec{x}\|} -  \frac{\vec{y}}{\|\vec{y}\|} \right\|^{2} \ \right] \ \right) } \ , 
   \end{align*}
        which is defined in terms of the norm, too.    
                             
    In this  paper we deal with generalized real normed vector spaces. We consider vector spaces   $X$  provided 
         with a { \it  `weight'}  or { \it `functional'} \ $\| \cdot \|$, that means we have a continuous  map \ 
         $\| \cdot \| : \ X  \longrightarrow \  \mathbbm{R}^{+}  \cup  \{ 0  \}$. \
         We assume that the weights are   { \it  `absolute homogeneous'}  or { \it `balanced'}, i.e. 
         $ \| r \cdot \vec{x} \| = |r| \cdot \|\vec{x}\| $  for  
         $ \vec{x} \in X ,  r \in   \mathbbm{R} $. \   We call such pairs \ $ ( X, \| \cdot \| ) $ \
         { \it  `balancedly weighted vector spaces'}, or for short \ \ $` \mathsf{BW} \ \text{spaces'} $\, . 
                       
    To avoid problems with a denumerator  $0$ we
         restrict our considerations to $ \mathsf{BW}$ spaces  which  are positive definite, i.e. 
         $\| \vec{x}\| = 0$ \ only for \ $ \vec{x} = \vec{0}$.
           
    Following the lines of an inner product       
    we define for each real number ${\bf \varrho}$ a continuous product \  
               $ < . \: | \: . >_{\bf \varrho} $  \ on $ X $. \ Let   $ \vec{x}, \vec{y}$ be two arbitrary elements 
               of $ X $. In the case of  \ $ \vec{x} = \vec{0} $ \ or \ $ \vec{y} = \vec{0}  $ \ we set \ 
               \ $ < \vec{x} \: | \:  \vec{y} >_{\bf \varrho} \ := \, 0 $, \ and if \ 
               $ \vec{x} , \vec{y} \neq \vec{0} $  (i.e.  $ \|\vec{x}\| \cdot \|\vec{y}\| > 0 $) 
      \ we define the real number   \\   \\
               \centerline { $ < \vec{x} \: | \:  \vec{y} >_{\bf \varrho}  \ \ :=  $  } 
     \begin{equation*}  
                 \|\vec{x}\| \cdot \|\vec{y}\| \cdot   \frac{1}{4} \cdot
                 \left[ \left\| \frac{\vec{x}}{\|\vec{x}\|}   +  \frac{\vec{y}}{\|\vec{y}\|} \right\|^{2}   - 
                 \left\| \frac{\vec{x}}{\|\vec{x}\|} - \frac{\vec{y}}{\|\vec{y}\|} \right\|^{2}  \right] \cdot
                 \left(   \frac{1}{4} \cdot 
                 \left[  \left\| \frac{\vec{x}}{\|\vec{x}\|}   +  \frac{\vec{y}}{\|\vec{y}\|} \right\|^{2}  + 
                 \left\| \frac{\vec{x}}{\|\vec{x}\|} - \frac{\vec{y}}{\|\vec{y}\|} \right\|^{2}   \right]
                                                        \right)  ^{\bf \varrho}  \ .
     \end{equation*}              
           
        It is easy to show that such product fulfils  the symmetry  ($ <\vec{x} | \vec{y}>_{\bf \varrho}$   
        $ \ = \ <\vec{y} | \vec{x}>_{\bf \varrho} $), the positive semidefiniteness 
        ($ <  \vec{x} | \vec{x} >_{\bf \varrho} \ \geq \ 0 $), and  the homogeneity 
        ($  < r \cdot \vec{x} | \vec{y} >_{\bf \varrho} \ = \ r \cdot <  \vec{x} | \vec{y} >_{\bf \varrho} $), 
        for \ $  \vec{x},  \vec{y} \in X $,  \  $ r \in \mathbbm{R}$.

   Let us fix  a  number ${\bf \varrho} \in  \mathbbm{R}$ and a positive definite 
        $ \mathsf{BW}$ space $ ( X, \| \cdot \| )$.  \  We are able to define for two vectors \ 
	      $ \vec{x},  \vec{y} \neq  \vec{0}  $ \ with the additional property \ \  
	      $ |< \vec{x} \: | \: \vec{y} >_{\bf \varrho}| \ \leq \ \|\vec{x}\| \cdot \|\vec{y}\| $ \ \
	      an  `angle'  $\angle_{\bf \varrho} (\vec{x}, \vec{y})$, 
	      according to the Euclidean angle in inner product spaces. \
	      Let  
	  \begin{equation*}   
	   \angle_{\bf \varrho}  (\vec{x}, \vec{y}) \ := \
           \arccos{\frac{< \vec{x} \: | \: \vec{y} >_{\bf \varrho}} { \|\vec{x}\| \cdot \|\vec{y}\| } }  \ .
    \end{equation*}  
                           
	  We consider mainly those pairs $ ( X, \| \cdot \|) $ \
	       where the triple \  $ \left( X , \| \cdot \|, \ < . \: | \: . >_{\bf \varrho}  \right) $  \ 
         satisfies the { \it Cauchy-Schwarz-Bunjakowsky Inequality } or { \sf CSB }inequality, \ that means  
         for all \  $ \vec{x}, \vec{y} \in X $ \  we have the inequality \ 
   $$ |< \vec{x} \: | \: \vec{y} >_{\bf \varrho} |  \ \leq \ \|\vec{x}\| \cdot \|\vec{y}\| \ , $$ 
        for a fixed  real number $ {\bf \varrho}$. 
        In this case  we get that the   `$\varrho$-{\it angle}'  $ \angle_{\bf \varrho} (\vec{x}, \vec{y}) $ \ 
        is defined for all \  $ \vec{x}, \vec{y} \neq \vec{0} $, and we shall express this by 
   $$ ` \ \text{The space} \ \left( X , \| \cdot \| \right) \ 
        \text{ has the angle } \  \angle_{\bf \varrho} \, \text{'} ,         $$ 
  This new `angle' has seven comfortable properties which are known from the Euclidean angle in inner product spaces,
     and it corresponds for all  ${\bf \varrho} \in  \mathbbm{R}$ to the Euclidean angle in the case that \
     $ \left( X , \| \cdot \| \right) $ \ already is an inner product space. 
                                              
   Let \  \ $ ( X , \| \cdot \| ) $ \ be a real positive definite  $ \mathsf{BW}$ space  with \ dim$(X) > 1 $.  
         Assume that the triple \  $ \left( X , \| \cdot \|, \ < . \: | \: . >_{\bf \varrho}  \right) $  \ 
         satisfies  the { \sf CSB }inequality for a fixed number $ {\bf \varrho} $. Hence we are able to define the
          $ {\bf \varrho} $-angle, and for elements \ $ \vec{x},\vec{y} \neq \vec{0}$ we have the properties that
    \begin{itemize}
          \item    \quad \ \ $\angle_{\bf \varrho} $  is a continuous surjective map 
                from \  \ $  (X \backslash \{\vec{0}\}) ^{2} $ \ onto the closed interval \ $ [0,\pi] $, 
          \item    \quad   \    \  $\angle_{\bf \varrho} (\vec{x}, \vec{x}) = 0 $,   
          \item    \quad  \   \ $\angle_{\bf \varrho}  ( -\vec{x}, \vec{x}) = \pi$,   
          \item    \quad   \  \  $\angle_{\bf \varrho}  (\vec{x}, \vec{y}) 
                = \angle_{\bf \varrho}  (\vec{y}, \vec{x}) $, 
          \item    \quad  \ \ For all \ $ r,s > 0 $ \
                 we  have  \ \ $\angle_{\bf \varrho} (r \cdot \vec{x}, s \cdot \vec{y}) 
                 = \angle_{\bf \varrho}  (\vec{x},\vec{y})$, 
          \item    \quad  \ \ $\angle_{\bf \varrho} ( - \vec{x}, - \vec{y}) 
          = \angle_{\bf \varrho}  (\vec{x},\vec{y})$,    
          \item    \quad \  $ \angle_{\bf \varrho} (\vec{x}, \vec{y}) + \angle_{\bf \varrho} (-\vec{x}, \vec{y}) 
                  = \pi $,    
    \end{itemize} 
      which are easy to prove.  \\
       
   Then we define some classes of real vector spaces. \
        Let  $ \mathsf{NORM}$ \ be the class of all real normed vector spaces. \
	      For all fixed real numbers ${\bf \varrho}$ let 
   $$ \mathsf{NORM}_{\bf \varrho} \ := \ \left\{ (X , \| \cdot \|) \in \mathsf{NORM} \ | \ 
               \text{ The normed space } (X , \| \cdot \|)  \text{ has the angle } \angle_{\bf \varrho} \right\}.  $$     
	We prove the statements  
		$$  \mathsf{NORM} \ = \ \mathsf{NORM}_{\bf \varrho}    $$
		             for all real numbers $\varrho$ from the closed interval \ $[-1 , 1 ]$,  \ and also  
	  $$ \mathsf{IPspace} \ = \ \bigcap_{ {\bf \varrho} \in { \bf \mathbbm{R}} } \ \mathsf{NORM}_{\bf \varrho} \ , $$
	  where \  $\mathsf{IPspace}$ denotes the class of all real inner product spaces. \ 
		Furher, if we assume four positive real  numbers \ $\alpha, \beta, \gamma , \delta $ \	such that  \\ 
	                \centerline{   $ -\delta < -\gamma < -1 <  1 < \alpha <  \beta $, }  \\
	  we obtain the chain of inclusions 
	  $$       \mathsf{NORM}_{\bf -\delta} \subset   \mathsf{NORM}_{\bf -\gamma} \subset \mathsf{NORM} 
	           \supset  \mathsf{NORM}_{\bf \alpha} \supset \mathsf{NORM}_{\bf \beta} \, . $$ 
	  We prove the inequalities \ \    
	  $$       \mathsf{NORM}_{\bf -\gamma} \neq  \mathsf{NORM} \neq  \mathsf{NORM}_{\bf \alpha}       $$      
	  and we strongly believe, but we have no proof that the  inclusions 
	     \ $    \mathsf{NORM}_{\bf -\delta} \subset   \mathsf{NORM}_{\bf -\gamma} $ \ and \ 
	       $    \mathsf{NORM}_{\bf \alpha} \supset \mathsf{NORM}_{\bf \beta}  $ \  are proper.  \\
	  
	 After that we return to the more general situation. We abandon the restriction of the triangle
	           inequality, again we consider  positive definite $ \mathsf{BW}$ spaces $ (X , \| \cdot \|) $, 
	           i.e. its weights $\| \cdot \|$  have to be positive definite   and absolute homogeneous only.
    We say \ { \it `positive definite balancedly weighted spaces' }  or \ $ \mathsf{pdBW}$ \ for the class
              of all such pairs, and we define  for all fixed real number  $ \varrho $ the class
   $$ \mathsf{pdBW}_{\bf \varrho} \ := \ \left\{ (X , \| \cdot \|) \in \mathsf{pdBW} \ | \ 
   \text{ The space} \ (X , \| \cdot \|)  \text{ has the angle }  \angle_{\bf \varrho} \right\}.  $$  
             We show \ \ 
             $  \mathsf{pdBW}_{\bf -1}  =  \mathsf{pdBW} $. \  Roughly speaking this means that for the angle 
             $ \angle_{\bf \varrho}$ the `best' choice is \ $ \bf \bf \varrho = -1 $, \ 
             since the angle  \ $ \angle_{\bf -1}$ \  is defined in every element of \ $ \mathsf{pdBW}$.  
                   
      For  real  numbers \ $\alpha, \beta, \gamma , \delta $ \	 with \
	           $ -\delta < -\gamma < -1  < \alpha <  \beta $  \ we get the inclusions   
   $$        \mathsf{pdBW}_{\bf -\delta} \subset   \mathsf{pdBW}_{\bf -\gamma} \subset  \mathsf{pdBW} 
             \supset  \mathsf{pdBW}_{\bf \alpha} \supset \mathsf{pdBW}_{\bf \beta} \, . $$            
      Since \ $ \mathsf{NORM}_{\bf -\gamma} \neq  \mathsf{NORM} $ \ we already know the fact \ 
	           $   \mathsf{pdBW}_{\bf -\gamma} \neq  \mathsf{pdBW} $.  \  Further we prove  \ 
	           $ \mathsf{pdBW} \neq \mathsf{pdBW}_{\bf \alpha} $, \ and we conjecture that the inclusions \  
             $        \mathsf{pdBW}_{\bf -\delta} \subset   \mathsf{pdBW}_{\bf -\gamma} $ \ and also \ 
             $        \mathsf{pdBW}_{\bf \alpha} \supset \mathsf{pdBW}_{\bf \beta}  $ \  are proper.  
                             
    To prove the above statements we define and use         
        `{\it convex corners}' which can occur even in normed vector spaces, and     
        `{\it concave  corners}' which can be vectors in $ \mathsf{BW}$ spaces which are not normed spaces. 
        Both expressions are mathematical descriptions of a geometric shape exactly what the names associate. 
        For instance, the well-known normed space \ $ (\mathbbm{R}^{2}, \| \cdot \|_{1}) $ with the norm \
        $ \| (x,y) \|_{1} = |x| + |y| $ \ has four convex corners at its unit sphere, they are just the  corners of the
        generated square.   
                    
    Further we introduce a function  $ \Upsilon $, 
  $$    \Upsilon:  \mathsf{pdBW} \longrightarrow \left[ -\infty, -1 \right] \times \left[ -1 , +\infty \right] \, , 
        \ \  \Upsilon( X , \| \cdot \| )  := ( \nu, \mu ) \ , $$
        which maps every real positive definite   $ \mathsf{BW}$ space to a pair of extended real  numbers 
        $ ( \nu , \mu )$, \  where
     \begin{align*}         
       & \nu  := \inf \{ {\bf \varrho} \in  \mathbbm{R} \ | \  ( X ,  \| \cdot \| ) \text 
       { has the angle }  \angle_{\bf \varrho} \}, \ \text{ and } \\ 
       & \mu  := \sup \{ {\bf \varrho} \in  \mathbbm{R} \ | \  ( X ,  \| \cdot \| ) \text 
       { has the angle }  \angle_{{\bf \varrho}}\}   \ .  
     \end{align*} 
   For an inner product space $( X , \| \cdot \| )$  we get immediately \ 
       $ \Upsilon( X , \| \cdot \| )  \ = \ ( -\infty , \infty ) $.   
   
   If  $( X ,  \| \cdot \| ) $ is an arbitrary space from the class  $\mathsf{pdBW}$ 
        with $ -\infty < \nu, \mu < \infty $, \ we show  that the infimum and the supremum are attained, i.e.
     \begin{align*}         
       &  \nu  = \min \{ {\bf \varrho} \in  \mathbbm{R} \ | \  ( X ,  \| \cdot \| ) \text 
       { has the angle }  \angle_{{\bf \varrho}}\}, \ \text{ and } \\ 
       &  \mu  = \max \{ {\bf \varrho} \in  \mathbbm{R} \ | \  ( X ,  \| \cdot \| ) \text 
       { has the angle }  \angle_{{\bf \varrho}}\}   \ .  
     \end{align*}       
        Let $ ( X ,  \| \cdot \| ) \in  \mathsf{NORM} $. We assume that $( X, \| \cdot \| )$ has a convex
        corner. \  Then we get  
   $$              \Upsilon( X ,  \| \cdot \| ) = ( -1 , 1 ) \ .  $$  
   For instance, for the normed space \ $ \left(\mathbbm{R}^{2}, \| \cdot \|_{1}\right) $ \ we have \ 
       $ \Upsilon( \mathbbm{R}^{2}, \| \cdot \|_{1}) = (-1,1) $.  \\
                                                    
   At the end we consider  products. \ For two spaces  
      $ (A , \| \cdot \|_A), (B , \| \cdot \|_B) \ \in  \mathsf{pdBW}$ we take its Cartesian product $ A \times B $,
      and we get a  set of balanced weights  $ \| \cdot \|_p $ on $ A \times B $, \ for \ $ p > 0 $. \
      For a positive number \ $ p $ for each element $ \left(\vec{a}, \vec{b} \right) \in A \times B$  we define 
      the non-negative  number
  $$ \left\|\left(\vec{a}, \vec{b}\right) \right\|_{p} \ 
                                            := \ \sqrt[p]{\|\vec{a} \|_{A}^{p} + \|\vec{b}\|_{B}^{p}} \ .  $$
    This makes the pair \ $ \left( A \times B, \left\| \, \cdot \, \right\|_{p} \right) $ \ to an element of   
    the class $ \mathsf{pdBW}$, and with this construction we finally ask two more interesting 
    and unanswered questions.  
     \newpage  
  \section{General Definitions} 
	            \ \ \  \ \ \ \       
    Let    \ $ X $ =  ($ X , \tau$) \  be an  arbitrary   real topological  vector  space,  \  that  means  that 
        the real  vector space   $ X $  is  provided with a topology  $ \tau $ such that  the addition  of two vectors
        and  the  multiplication  with real numbers  are continuous.  Further let \  $\| \cdot \|$ \ denote a 
        {\it positive functional} or a   {\it weight}  on $ X$, that means that \  
        $\| \cdot \|$: \ $ X \longrightarrow$  $\mathbbm{R}^{+}  \cup  \{ 0  \}$ \  is  continuous, the non-negative 
        real numbers \ $\mathbbm{R}^{+} \cup \{0\}$ \ carry  the  usual  Euclidean topology. 
                                                        
             We consider some  conditions.       \\  
       $\widehat{(1)}$:  \    For  all \ $ \it r \  \in \ \mathbbm{R}$ \ \rm \ and  all   \ \ 
                   \it   $\vec{x}  \in X$ \  \rm we have: \
                   \it  $\|r \cdot  \vec{x}\|$ =   $|r|$ $ \cdot  \|\vec{x}\|$ \quad  \hfill  \rm (`absolute  
                       homogeneity'),    \\  
       $\widehat{(2)}$:  \    \rm  $\| \vec{x}\| = 0$ \quad  { \rm if and only if} \quad  $ \vec{x} = \vec{0}$  
                   \hfill  \quad \rm (`positive \ definiteness'), \\
       $\widehat{(3)}$:  \    \rm  For \  \it  $\vec{x} , \vec{y} \in X$ \ \rm \  it  holds \  
                   \it   $\|\vec{x}+\vec{y}\| \leq  \|\vec{x}\|  + \|\vec{y}\|$ 
                                                             \hfill \quad \rm (`triangle \ inequality'),   \\
       $\widehat{(4)}$: \     \rm  For \  \it  $\vec{x} , \vec{y} \in X$ \ \rm  \  it holds \ 
                   \it   $\|\vec{x}+\vec{y}\|^{2} +   \|\vec{x}-\vec{y}\|^{2}  =   2 \cdot [  \|\vec{x}\|^{2}  
                   + \|\vec{y}\|^{2}  ] $     \hfill \ \rm (`parallelogram identity').   \\
   $\begin{array}[ ]{lll}  \rm   
        \rm If \   \| \cdot \|  \  fulfils \ \widehat{(1)} , & \text{then we call} \ \| \cdot \| \   
                                             \text{a {\it balanced  weight}  on} \ {\it X},  \\  
        \rm if \   \| \cdot \|  \  fulfils \ \widehat{(1)} , \widehat{(3)} & \text{then} \ \| \cdot \| \ \
                \text{is  called a { \it seminorm } on { \it X}, \  and }  \\ 
        \rm if \   \| \cdot \|  \  fulfils \ \widehat{(1)} , \widehat{(2)} , \widehat{(3)} &  
                   \text{then}  \ \| \cdot \| \ \text{is  called a { \it norm } on \ {\it X}, \ and }     \\ 
        \rm if \   \| \cdot \|  \  fulfils \ \widehat{(1)}, \widehat{(2)}, \widehat{(3)}, \widehat{(4)}
                   & \text{the pair} \  (X, \| \cdot \|) \  \text{is called an {\it inner product space}} . 
   \end{array} $  \\
      According to this four cases we call                                                           
         the pair \ ($ X , \| \cdot \|$) \ a { \it balancedly weighted vector space}  (or $ \mathsf{BW}$ space), 
         a  { \it seminormed vector space}, a  { \it normed vector space},  or an 
         { \it inner product space} (or ${ \mathsf{IP}}$ space),  respectively.  See also the interesting paper
         {\bf \cite{RubinStone}} where it has been shown that \  $\widehat{(1)},  \widehat{(2)}, \widehat{(4)} $ 
         is sufficient to get  $\widehat{(3)}$, and therefore to get an inner product space. \
                 
      We  shall  restrict our considerations mostly to $ \mathsf{BW}$ spaces which  are positive definite, i.e. 
         $\| \vec{x}\| = 0$  only for $ \vec{x} = \vec{0}$,  i.e. they satisfy  $\widehat{(2)}$.  
         That means in a pair \ ($ X , \| \cdot \|$) \ the weight \ $ \| \cdot \|$ \  is always positive definite,
         except we say explicitly the contrary.   
    \begin{remark}   \rm
         In a positive definite  $ \mathsf{BW}$ space \ $( X, \| \cdot \|$) \ we can generate a 
         `distance' \ $d$\  by \ $d(\vec{x}, \vec{y}) :=  \|\vec{x}- \vec{y}\| $. \ Note that generally the pair
         \ $( X , d ) $ is not a metric space.  
    \end{remark}       
             
  Now let \quad   $ < . \:  | \: . > \ : \  X^{2} \longrightarrow   \mathbbm{R} $ \  be continuous 
	      as  a  map  from  the product space  $ X \times X $ \ into the Euclidean space  \ $ \mathbbm{R}. $    
	 
	 We  consider some  conditions.  \\
	 $\overline{(1)}$: \ \   For  all  $ \it r  \in  \mathbbm{R}$ \ \rm  and   
                    \it  $\vec{x}, \vec{y}   \in X $ \  \rm  it holds  \
                    \it  $ <r \cdot \vec{x} \ | \ \vec{y}> \ = \ r \cdot <\vec{x} \ | \ \vec{y}>$    
                     \   $   { \ } $  {  $ \ $ }         \hfill \rm     (`homogeneity'),   \\
   $\overline{(2)}$:  \  \  \rm  For \ all \ \ \it  $\vec{x} , \vec{y} \in X$ \ \rm  it holds \  \  \ 
                     \it  $ <\vec{x} \ | \ \vec{y}>  \  = \ <\vec{y} \ | \ \vec{x}>   $ 
                                                     \hfill \quad \rm (`symmetry'),   \\    
   $\overline{(3)}$: \ \    For  all \  $  \vec{x}  \in X$ \  \rm  we have \
                    \it  $ <\vec{x} \ | \ \vec{x}> \  \  \geq \  0  $
                                                      \hfill  \quad \rm  (`positive semidefiniteness'),       \\           $\overline{(4)}$: \ \  \rm  $  <\vec{x} \ | \ \vec{x}> \  =  0  $ \quad  { \rm if and only if} 
                        \quad  $ \vec{x} = \vec{0}$  
                                                      \hfill  \quad \rm (`definiteness'), \\
   $\overline{(5)}$:  \  \  \rm  For \ all \ \it  $\vec{x} , \vec{y} ,  \vec{z} \in X$ \ \rm  it holds \quad  
                 $<\vec{x} \ | \ \vec{y}+\vec{z}> \ = \    <\vec{x} \ | \ \vec{y}>  + <\vec{x} \ | \ \vec{z}>$  
                   \hfill \ \rm (`linearity').
   $ \begin{array}[ ]{l}  \rm                               
            \text{If} \   < . \: | \: . > \  \text{fulfils} \ \overline{(1)},\overline{(2)},\overline{(3)} , \  
            \text{we  call  \ $< . \: | \: . >$ \  a {\it homogeneous product}  on } {\it X},  \\ 
            \text{if} \   < . \: | \: . >   \  \text{fulfils}  \overline{(1)},\overline{(2)},\overline{(3)}, 
            \overline{(4)},\overline{(5)}, \ \text{then} \ < . \: | \: . > \ 
            \text {is  an { \it inner  product}  on  {\it X}}.    
      \end{array}    $   \\
   According to these  cases we call                                                           
            the pair \ ( $ X , < . \: | \: . >$ ) \ a  { \it homogeneous product vector space}, 
            or an  { \it inner product space } (or ${\mathsf  {IP}}$ space),  respectively.       
   \begin{remark}  \rm
      We use the term  \ `${\mathsf  {IP}}$ space' \ twice, but both definitions coincide: It is well-known that a  norm 
         is based on an inner product   if and only if the parallelogram identity holds. 
   \end{remark}                                   
   Let \   $\| \cdot \|$ \  be denote a positive functional on $ X$. \  Then we define two closed subsets of   $ X$: \\
           $ {\bf S} \ := {\bf S}_{(X,\| \cdot \|)} \ :=  \  \{ \: \vec{x} \in X \ |  \ \|\vec{x}\| = 1 \: \}$, \
           the {\it unit sphere} of $X$, \\
           $ {\bf B} \ := {\bf B}_{(X,\| \cdot \|)} \ :=  \  \{ \: \vec{x} \in  X \ |  \ \|\vec{x}\| \leq 1 \: \} $, \
           the {\it unit ball} of $X$. 
   
    Now assume that the real vector space \ $X$ \ is provided with a positive functional \ $ \| \cdot \|$ \ and a
           product  \  $ < . \: | \: . > $.  \  Then the triple \  $( X , \| \cdot \| ,  < . \: | \: . > ) $ \ 
           satisfies the { \it Cauchy-Schwarz-Bunjakowsky Inequality } or { \sf CSB }inequality \ if and only if \
           for all \ $ \vec{x}, \vec{y} \in X $ \ we have the inequality \\
        \centerline{   $ |< \vec{x} \: | \: \vec{y} >| \ \leq \ \|\vec{x}\| \cdot \|\vec{y}\| $.  }
                                       
    Let  \ $ A $ \ be an arbitrary subset  of a  real vector space \ $ X $.  \ 
            Let $ A $ \  has the  property  that 
            for  arbitrary \ $ \vec{x} , \vec{y}  \in A $ \ and for  every  \ $ 0 \leq t \leq 1 $ \ we have \ 
	          $ t \cdot \vec{x} + (1-t) \cdot \vec{y} \in A. $  \ Such a  set $ A $ is called  { \it  convex}.
	          The unit ball \ $ {\bf B} $ \ in a seminormed space is convex because of the triangle inequality.  
	                                   
    A convex set  $ A $ is called { \it strictly convex} if and only if for each number $ 0 < t < 1 $ \ 
            it holds that the linear combination \ $ t \cdot \vec{x} + (1-t) \cdot \vec{y}$ \ lies in the interior of 
            $ A $, for all vectors $ \vec{x},  \vec{y} \in A $.           
	%
                                            
    For two real numbers \ $ a < b $ \ the term \ $[a, b]$ \ means the closed interval of \ $a$ and $b$, \
         while  $(a,b)$ \ means the pair of two numbers or the open interval between  $a$ and $b$.    
   \section{Some Examples of Balancedly Weighted Vector Spaces}   \qquad  \label{section three}
   We describe some easy examples of balanced weights on the usual vector space $ \mathbbm{R}^{2}$. 
   At first, for each  $ p \in  \{\infty, -\infty \} \cup \mathbbm{R} $ \ we construct a 
   balanced weight  $\| \cdot \|_{p}$ on $ \mathbbm{R}^{2}$.  
   For  $\vec{x} = (x , y ) \in  \mathbbm{R}^{2} $ \ for a  real number \ $ p > 0 $ \ we set  
   $\|\vec{x}\|_{p}  := \sqrt[p]{|x|^{p}+|y|^{p}}$, \
   and   for the negative real number  $-p$  we define correspondingly 
    \rm  \[   \|\vec{x}\|_{-p}  :=             
             \begin{cases}
           \sqrt[-p]{|x|^{-p}+|y|^{-p}}   &    \quad \mbox{if} \quad       x \cdot y \neq 0   \\    
            \ \                       0   &    \quad \mbox{if} \quad       x \cdot y  = 0   \\            
                         \end{cases}        \]   
    \rm    and let \  \ $\|\vec{x}\|_{\infty} := \max \{ |x|\: , |y| \} , $   \ \
          and  \  \ $\|\vec{x}\|_{-\infty} := \min \{ |x|\: , |y|  \} $. \  For convenience, we define 
     for \ $ p = 0 $ \ the trivial seminorm \  $\|\vec{x}\|_{0} = 0$ \ for all \ $\vec{x} \in  \mathbbm{R}^{2} $. \\     
     These  weights  $\| \cdot \|_{p}$  are called  the { \it H\"older weights }  on   $\mathbbm{R}^{2}$.  
           Since   $ \| \cdot \|_{p}$ fulfils $\widehat{(1)}$ , the pair  $ ( \mathbbm{R}^{2} , \| \cdot \|_{p})$ 
           is a   $ \mathsf{BW}$ space for each $p \in \{\infty, -\infty \} \cup \mathbbm{R} $.
  \begin{remark}  \rm  The above definition for negative real numbers $-p$ may be  clearer if one  notes  
        $$  \|\vec{x}\|_{-p} \ = \ \sqrt[-p]{|x|^{-p}+|y|^{-p}} 
            \ = \ \left[ \sqrt[p]{\frac{1}{|x|^{p}}+\frac{1}{|y|^{p}}} \ \right]^{-1} 
            \ = \ \frac{ |x| \cdot |y|} { \sqrt[p]{|x|^{p}+|y|^{p}}} 
             \qquad \text{if} \quad x \cdot y \neq 0 \ .                $$
        For $ -p < 0$  \ we have \  $\| (x,y) \|_{-p} =  0 $ \ if and only if   $x = 0$ or $y = 0$.
        That means \ $ (x,y) $ lies on one of the two axes.    
	      Furthermore, for arbitrary  numbers $ p \in \{\infty, -\infty \} \cup \mathbbm{R} $, 
        we have that the pair $( \mathbbm{R}^{2},  \| \cdot \|_{p}) $ is a normed space if and only if $ p \geq 1$, 
        (the {\it H\"older norms} ), and the weight $\| \cdot \|_{p}$ is positive definite if and only if \ $p > 0 $. 
        For  $ p  = 2$ we get the usual and well-known Euclidean norm  $\| \cdot \|_{Euclid}$ . 
   \end{remark}  
    Other interesting non-trivial examples are the following, the first is not positive definite.  \ \
	 Let \  $(\mathbbm{R}^{2},\| \cdot \|_A ) $ \  be a $ \mathsf{BW}$ space with the unit sphere   
	            $ { \bf S}_A $, the set of unit vectors, and define   \\
	  	   \centerline{   $ {\bf S}_A 
	           \ := \ \left\{ (x,y) \in 	\mathbbm{R}^{2} \ | \ |x| \cdot |y| = 1 \right\} $, }
	    and extend the weight  $ \| \cdot \|_A$ by homogeneity. Every point on the axes has the weight zero.
	                                       
	 The next two examples are positive definite, but weird. 
	   At first we define a weight $ \| \cdot \|_B$ \ on  $\mathbbm{R}^{2}$ by defining the unit sphere  ${\bf S}_B $, \\ 
	      \centerline{  $ {\bf S}_B  \ := \ \left\{ (x,y) \in 	\mathbbm{R}^{2} \ | \ 
	                            \sqrt{|x|^{2}+|y|^{2}} = 1 \ \wedge (x,y) \notin \{(1,0),(-1.0)\}  \right\}
	                           \ \bigcup  \ \left\{ (2,0), (-2,0)   \right\}  $, }    \\  
	    and we extend the weight  $ \| \cdot \|_B$ by homogeneity.     \\
	    In a similar way the next weight  $ \| \cdot \|_C$  is constructed by fixing the unit sphere  ${\bf S}_C $, \\
	       \centerline{  $ {\bf S}_C  \ := \ \left\{ (x,y) \in 	\mathbbm{R}^{2} \ | \ 
	                 \sqrt{|x|^{2}+|y|^{2}} = 1 \ \wedge (x,y) \notin \{(1,0),(-1.0)\}  \right\}
	                 \ \bigcup \ \left\{ \left(\frac{1}{2},0\right), \left(-\frac{1}{2},0\right) \right\}$, }  \\  
	    and  extending the weight  $ \| \cdot \|_C$ by homogeneity.     \\
	    The pairs $ \left( \mathbbm{R}^{2},\| \cdot \|_B \right) $  and $ \left( \mathbbm{R}^{2},\| \cdot \|_C \right)$
	    are positive definite $ \mathsf{BW}$ spaces.  \\
	                         
	    Some definitions and discussions to separate these strange examples  
	    $ \left(\mathbbm{R}^{2},\| \cdot \|_B \right) $ 
	    and    $ \left( \mathbbm{R}^{2},\| \cdot \|_C \right) $   from the others  would be desirable. 
             
   \section{On Angle Spaces} 
             \qquad                   
  	In the usual Euclidean plane  $\mathbbm{R}^{2}$ angles are  considered for more than 2000 years. With the idea 
  	       of    `metrics' and `norms' others than the Euclidean one  the idea came to have also 
	         orthogonality and  angles in generalized metric and normed spaces, respectively. 
	         The first attempt to define a 
	         concept of generalized `angles' on metric spaces was made by Menger  { \bf \cite{Menger}}, p.749 .
        	 Since then a  few ideas have been developed, see the references { \bf \cite{Diminnie/Andalafte/Freese1}},
	         { \bf \cite{Diminnie/Andalafte/Freese2}},  { \bf \cite{Gunawan}}, { \bf \cite{Milicic1}},
	         { \bf \cite{Milicic2}}, { \bf \cite{Milicic3}},  { \bf \cite{Singer}}, {\bf \cite{J.E.Valentine/C.Martin}},
	         { \bf \cite{J.E.Valentine/S.G.Wayment}} . 
	  In this paper we focus our attention on  real normed spaces as a generalizitation of real inner product spaces. \  
	         Let  $ \left(  X , < . \: | \: . > \right)$  be an  ${\mathsf  {IP}}$ space, and let \ $ \| \cdot \| $ \ be 
	         the associated  norm,  $ \|\vec{x}\| := \sqrt{< \vec{x} | \vec{x} >} $, 	 then 	the  triple 
	         $\left( X ,  \| \cdot \|  ,  < . \: | \: . >  \right) $ \  fulfils the  { \sf CSB } inequality, 
	         and we have  for all \  $ \vec{x},\vec{y} \neq \vec{0} $ \ the well-known  Euclidean angle   \ \   
           $\angle_{Euclid} (\vec{x}, \vec{y}) := 
           \arccos{\frac{< \vec{x} \: | \: \vec{y} >}{ \|\vec{x}\| \cdot \|\vec{y}\|}}$  \ \ 
           with all its comfortable properties \ ({\tt An} 1) - ({\tt An} 11).   
  	\begin{definition}    \rm      \label{Definition eins}
           Let  ($ X , \| \cdot \|$)  be a real positive definite $ \mathsf{BW}$ space.
           We call  the triple $ \left(X , \| \cdot \|,\angle_X \right)$  an  { \it angle space }  if and only if 
           the following seven conditions ({\tt An} 1)  $-$  ({\tt An} 7)  are fulfilled.
           We regard vectors \ $ \vec{x}, \vec{y} \neq \vec{0} $. 
    \begin{itemize}
        \item  ({\tt An} 1) \quad $\angle_{X}$ \ \ is a continuous  map  
                  from \ $  (X \backslash { \{\vec{0}\}}) ^{2} $ \  into the interval \ $ [0,\pi] $. \
                  We say that \\
                  \centerline{ $ `\left( X, \| \cdot \| \right) \text{ has the angle } \ \angle_{X} \, \text{'} . $ } 
        \item    ({\tt An} 2) \quad  For all \ \ $ \vec{x} $ \ we have 
                  \  $\angle_{X}(\vec{x}, \vec{x}) = 0 $.   
        \item     ({\tt An} 3) \quad  For all \ \ $ \vec{x} $ \ we have 
                  \ $\angle_{X} ( -\vec{x}, \vec{x}) = \pi$.   
        \item  ({\tt An} 4) \quad  For all \ \ $ \vec{x},\vec{y} $ \ 
                  we have \  $\angle_{X} (\vec{x}, \vec{y}) = \angle_{X} (\vec{y}, \vec{x}) $. 
        \item   ({\tt An} 5) \quad  For all \ $ \vec{x},\vec{y} $ \ and all \ $ r,s > 0 $ \ 
                  we  have  \    
                  $\angle_{X} (r \cdot \vec{x}, s \cdot \vec{y}) = \angle_{X} (\vec{x},\vec{y})$.
        \item   ({\tt An} 6) \quad For all \  $ \vec{x},\vec{y}  $ \  \  we  have 
                  \  $\angle_{X} ( - \vec{x}, - \vec{y}) = \angle_{X} (\vec{x},\vec{y})$.          
        \item   ({\tt An} 7) \quad   For all \  $ \vec{x},\vec{y} $ \ \   we  have 
                  \  $ \angle_{X}(\vec{x}, \vec{y}) + \angle_{X}(-\vec{x}, \vec{y})  = \pi $.               
      \end{itemize}  
    \end{definition}             
    Furthermore we write down some more properties which seem to us `desirable', but `not absolutely necessary'.       
   	\begin{itemize}
         \item   ({\tt An} 8) \quad  For all \ $\vec{x},\vec{y}, \ \vec{x}+\vec{y} \in X \backslash {\{\vec{0}\}}$ \ 
                  we    have    \  $\angle_{X}(\vec{x}, \vec{x}+\vec{y}) + 
                  \angle_{X}(\vec{x}+\vec{y},\vec{y}) =   \angle_{X}(\vec{x},\vec{y})$. 
         \item  ({\tt An} 9) \quad For all $\vec{x},\vec{y}, \ \vec{x}-\vec{y} \in X \backslash {\{\vec{0}\}}$  \ 
                  we  have    \ $ \angle_{X}(\vec{x},\vec{y})  + \angle_{X}(-\vec{x},
                  \vec{y}-\vec{x}) +  \angle_{X}(-\vec{y},\vec{x}-\vec{y})  = \pi $. 
         \item  ({\tt An} 10)  \quad For all \ $\vec{x},\vec{y}, \ \vec{x}-\vec{y} \in X \backslash {\{\vec{0}\}}$ \ 
                  we  have  \ $\angle_{X}(\vec{y}, \vec{y}-\vec{x})
                  +  \angle_{X}(\vec{x},\vec{x}-\vec{y}) =  \angle_{X}(-\vec{x},\vec{y}) $.  
         \item ({\tt An} 11) \quad  For any two  linear independent vectors \ 
                $ \vec{x},\vec{y}  \in X $ \ we  have  a  decreasing  homeomorphism \\
                 \centerline{  $ \Theta :  \mathbbm{R} \longrightarrow  (0,\pi) , \ \
                 t \mapsto \angle_{X}(\vec{x}, \vec{y} + t \cdot\vec{x})$. }     
    \end{itemize}              
    \begin{remark}  \rm   \label{remark3} 
         We add another condition. \
          If \ $( Y , \| \cdot \| )$ \ is an \ ${\mathsf  {IP}}$ space,  and if the triple 
          $ \left(Y , \| \cdot \|,\angle_Y \right)$ is an  { `angle space' }, then it should hold that \ \ 
         $\angle_{Y} = \angle_{Euclid} $, \ i.e. the new angle should be a generalization of the Euclidean angle. 
    \end{remark}
    \newpage      
  	\section{An Infinite Set of Angles}  \label{section fuenf}	  \qquad                           
       Now assume that the real topological vector space \ $( X , \tau ) $ \ is provided with a positive 
               functional \ $ \| \cdot \|$ \ and a  product  \  $ < . \: | \: . > $.   Take two elements 
               $ \vec{x}, \vec{y} \in X $ with the two properties $ \|\vec{x}\| \cdot \|\vec{y}\| \neq 0 $ \ and \
               $ |< \vec{x} \: | \: \vec{y} >|  \leq  \|\vec{x}\| \cdot \|\vec{y}\| $.   \  
               Then we could define an angle between these two elements, \ \ $ \angle (\vec{x}, \vec{y}) :=
               \arccos{\frac{< \vec{x} \: | \: \vec{y} >}{ \|\vec{x}\| \cdot \|\vec{y}\| } }. $ \ \
               If  the triple \  $\left( X , \| \cdot \| ,  < . \: | \: . > \right) $ \ satisfy the 
               { \it Cauchy-Schwarz-Bunjakowsky Inequality } or { \sf CSB }inequality, then we would be
               able to define  for all $ \vec{x}, \vec{y} \in X $ with $ \|\vec{x}\| \cdot \|\vec{y}\| \neq 0 $ \
               this angle \  $ \angle_{ } (\vec{x}, \vec{y}) := 
               \arccos{\frac{< \vec{x} \: | \: \vec{y} >}{ \|\vec{x}\| \cdot \|\vec{y}\| }} \ \in [0, \pi]$. 
  
     Let  the pair \ $( X , \| \cdot \| ) $  be a real  $ \mathsf{BW}$ space, i.e. the weight \ $\| \cdot \|$
               is absolute  homogeneous, or { \it `balanced'}. \ Notice again that we only deal with positive definite
               weights  $ \| \cdot \| $. 
   \begin{definition}  \rm   \label{naechste Definition}  
     Let  ${\bf \varrho}$  be an arbitrary real number. We define a continuous product \ \ 
               $ < . \: | \: . >_{\bf \varrho} $  \ on $ X $. \ Let  \ $ \vec{x}, \vec{y}$ be two arbitrary elements 
               of $ X $. In the case of  $ \vec{x} = \vec{0}$  or  $ \vec{y} = \vec{0}$  we set \ 
               $ < \vec{x} \: | \:  \vec{y} >_{\bf \varrho} \ := \, 0 $,  and if  \ 
               $ \vec{x}, \vec{y} \neq \vec{0} $ (i.e. $ \|\vec{x}\| \cdot \|\vec{y}\|  \neq  0 $) \
               we define the real number   \ $ < \vec{x} \: | \:  \vec{y} >_{\bf \varrho}  \ :=  $   
     \begin{equation}  
                 \|\vec{x}\| \cdot \|\vec{y}\| \cdot   \frac{1}{4} \cdot
                 \left[ \left\| \frac{\vec{x}}{\|\vec{x}\|}   +  \frac{\vec{y}}{\|\vec{y}\|} \right\|^{2}   - 
                 \left\| \frac{\vec{x}}{\|\vec{x}\|} - \frac{\vec{y}}{\|\vec{y}\|} \right\|^{2}  \right] \cdot
                 \left(  \frac{1}{4} \cdot 
                 \left[  \left\| \frac{\vec{x}}{\|\vec{x}\|}   +  \frac{\vec{y}}{\|\vec{y}\|} \right\|^{2}  + 
                 \left\| \frac{\vec{x}}{\|\vec{x}\|} - \frac{\vec{y}}{\|\vec{y}\|} \right\|^{2}   \right]
                                                        \right) ^{\bf \varrho} 
     \end{equation}               $ {  }  $       \hfill  $\Box$      
   \end{definition}
       For the coming discussions it  is very useful  to introduce some abbreviations. We define for arbitrary vectors 
       \ $ \vec{x}, \vec{y} \neq \vec{0}$ \ (hence  \ $ \|\vec{x}\| \cdot \|\vec{y}\|  \neq  0 $) \ 
       two non-negative real numbers  $ { \bf s}  $  and  $ { \bf d}  $, 
   $$ { \bf s}  \ :=  \  { \bf s}( \vec{x}, \vec{y}) 
       \ :=  \  \left\| \frac{\vec{x}}{\|\vec{x}\|}   +  \frac{\vec{y}}{\|\vec{y}\|} \right\| , \quad \text{and} \quad
       { \bf d} \ := \  { \bf d}( \vec{x}, \vec{y})  
       \ :=  \   \ \left\| \frac{\vec{x}}{\|\vec{x}\|}   -  \frac{\vec{y}}{\|\vec{y}\|} \right\| \ ,      $$
       and also  two real numbers \ $  { \bf \Sigma} $ and  ${ \bf \Delta} $, the latter can be negative, \\
     \centerline {  $  { \bf \Sigma} := { \bf \Sigma}( \vec{x}, \vec{y}) \ := \ { \bf s}^{2} + { \bf d}^{2} \ = \
                        \left\| \frac{\vec{x}}{\|\vec{x}\|}   +  \frac{\vec{y}}{\|\vec{y}\|} \right\|^{2} 
                 +      \left\| \frac{\vec{x}}{\|\vec{x}\|}   -  \frac{\vec{y}}{\|\vec{y}\|} \right\|^{2} $ , } \\
      and \\ 
     \centerline {  $  { \bf \Delta} \ := \ { \bf \Delta}( \vec{x}, \vec{y})  \ := \ { \bf s}^{2} - { \bf d}^{2} \ = \
                 \left\| \frac{\vec{x}}{\|\vec{x}\|}   +  \frac{\vec{y}}{\|\vec{y}\|} \right\|^{2} 
                 -      \left\| \frac{\vec{x}}{\|\vec{x}\|}   -  \frac{\vec{y}}{\|\vec{y}\|} \right\|^{2} $ . }  \\
     All  defined four variables depend on two vectors \ $ \vec{x}, \vec{y} \neq \vec{0} $.  
     Since  \ $( X , \| \cdot \| )$ \ is positive definite,  $  { \bf \Sigma}$ must be a positive number. \ 
     Note  the inequality \  
	   $ 0 \leq  \left| { \bf \Delta}  \right| \leq  {\bf \Sigma} $.   
                                  
  With this abbreviations the above formula of the product  \ $ < \vec{x} \: | \:  \vec{y} >_{\bf \varrho}$  \ 
     is shortened to                
    \begin{equation*}   
       < \vec{x} \: | \:  \vec{y} >_{\bf \varrho}   \ \ =  \ \            
           \begin{cases}     0  \quad  &  \  \mbox{for} \  \vec{x} = \vec{0} \ \text{ or }\  \vec{y} = \vec{0} \  \\   
                      \|\vec{x}\| \cdot \|\vec{y}\| \cdot \frac{1}{4} \cdot  { \bf \Delta} \cdot
                      \left(  \frac{1}{4} \cdot  { \bf \Sigma}   \right)  ^{\bf \varrho} 
                              \quad     & \  \mbox{for} \   \vec{x}, \vec{y} \ \neq  \vec{0}  \ , \\            
           \end{cases}  \qquad \text{for all } \ {\bf \varrho} \in \mathbbm{R}. 
   \end{equation*} 
   \begin{lemma} \label{lemma 1} 
    In the case that \  $ \left( X , \| \cdot \| \ \right) $ is already an ${\mathsf  {IP}}$ 
       space with the inner product 
       $ < . \: | \:  . >_{IP} $, then  our product corresponds to the inner product, i.e. for all 
       \ $ \vec{x}, \vec{y} \in X $  \ we have
     \begin{align}  
             < \vec{x} \: | \:  \vec{y} >_{IP} \ \ = \ \ < \vec{x} \: | \:  \vec{y} >_{\bf \varrho} \qquad
                \text{ for all} \  {\bf \varrho} \in \mathbbm{R}.  
      \end{align}            
   \end{lemma}  
   \begin{proof}
         Let the  normed space  \ $ \left( X , \| \cdot \| \ \right) $ \ be an inner product space or 
         `${\mathsf  {IP}}$ space'. 
         For two elements  $ \vec{x}, \vec{y} \in X $  \ we can express its inner product  by its norms, i.e. we have  
      \begin{equation*}  
           < \vec{x} \ | \ \vec{y} >_{IP}  \ \ =  \ 
          \frac{1}{4} \cdot \left( \: \|\vec{x}+\vec{y}\|^{2} -   \|\vec{x}-\vec{y}\|^{2} \: \right) ,
      \end{equation*}  
         and by the homogeneity and the symmetry of the inner product we can write for  \ 
         $ \vec{x}, \vec{y} \neq \vec{0}$ 
      \begin{equation*}  
          < \vec{x} \ | \ \vec{y} >_{IP}  \ =  \ 
          \|\vec{x}\| \cdot \|\vec{y}\| \cdot < \frac{\vec{x}}{\|\vec{x}\|} \ | \ \frac{\vec{y}}{\|\vec{y}}\| >_{IP}  
          \  =  \  \|\vec{x}\| \cdot \|\vec{y}\| \cdot  \frac{1}{4} \cdot 
          \left[  \ \left\| \frac{\vec{x}}{\|\vec{x}\|}   +  \frac{\vec{y}}{\|\vec{y}\|} \right\|^{2} \  - \
          \left\| \frac{\vec{x}}{\|\vec{x}\|}   -   \frac{\vec{y}}{\|\vec{y}\|} \right\|^{2} \  \right] \ .     
       \end{equation*}  
       This shows the equation \
       $ < \vec{x} \: | \:  \vec{y} >_{IP} \ = \ \|\vec{x}\| \cdot \|\vec{y}\| \cdot \frac{1}{4} \cdot {\bf \Delta}$. \        Further, in inner product spaces the parallelogram identity holds, that means  for unit vectors 
       $ \vec{v} \text { and } \vec{w} $ \ we have   
       $$  \|\vec{v}+\vec{w}\|^{2} +   \|\vec{v}-\vec{w}\|^{2} \ = \ 
                            2 \cdot \left( \|\vec{v}\|^{2} + \|\vec{w}\|^{2} \right) \ =  \  4 \ . $$ 
       Then it follows for the unit vectors \ $  \frac{\vec{x}}{\|\vec{x}\|} $ and  $ \frac{\vec{y}}{\|\vec{y}\|}$ \
   $$  \frac{1}{4} \cdot \left( \: \left\| \frac{\vec{x}}{\|\vec{x}\|} + \frac{\vec{y}}{\|\vec{y}\|} \right\|^{2} 
                 +  \left\| \frac{\vec{x}}{\|\vec{x}\|} -  \frac{\vec{y}}{\|\vec{y}\|} \right\|^{2} \right) 
        \ = \  \frac{1}{4} \cdot ( { \bf s}^{2} + { \bf d}^{2} ) \ = \  \frac{1}{4} \cdot  { \bf \Sigma} \ = \ 1 \ , $$
        and the lemma is proven.       
    \end{proof}    
    \begin{lemma} \label{lemma 2} 
          \it  For a positive definite $ \mathsf{BW}$ space $ ( X , \| \cdot \| \   ) $  
          the pair ( $ X , < . \: | \: . >_{\bf \varrho}  $ )  is a homogeneous product vector space,  with 
          $ \|\vec{x}\| =   \sqrt{ <\vec{x} \: | \: \vec{x}>_{\bf \varrho}  } $,  \ for all \ $ \vec{x} \in X $ 
          and for all  real numbers ${\bf \varrho}$.  
     \end{lemma}  
     \begin{proof}  \quad                                   
             We have  \ $ < . \, | \, . >_{\bf \varrho} \, : \,  X^{2} \longrightarrow    \mathbbm{R}  $,  \ 
             and the properties  $\overline{(2)}$ (symmetry) \ and $\overline{(3)}$ (positive semidefiniteness) are
             rather trivial.  Clearly,  $ \|\vec{x}\| = \sqrt{ <\vec{x} \: | \: \vec{x}>_{\bf \varrho} } \ $ for all 
             $ \vec{x} \in X $.    We  show \   $\overline{(1)}$,  the homogeneity. 
             For a  real number  \ $ r > 0 $ \ \ holds  that  \ \ 
             $ <r \cdot \vec{x} \ | \ \vec{y}>_{\bf \varrho}  \ = \ r \cdot <\vec{x} \ | \ \vec{y}>_{\bf \varrho} $, 
             because \  $ ( X , \| \cdot \| \   ) $  \  satisfies \ $\widehat{(1)}$ .    \ \
             Now we prove \ \ $ <- \vec{x} \ | \ \vec{y}>_{\bf \varrho} \ = \ - <\vec{x} \ | \ \vec{y}>_{\bf \varrho}$.
             \   Let \ $ \vec{x}, \vec{y} \neq \vec{0}$ .  \  Note that the factor $ { \bf \Sigma} $
             is not affected by a negative sign at $ \vec{x}$ or $ \vec{y}$ . \    We  have    
        \begin{eqnarray*}            
               - <\vec{x} \ | \ \vec{y}>_{\bf \varrho}     & = &   
               - \  \frac{1}{4} \cdot  \|\vec{x}\| \cdot \|\vec{y}\| \cdot 
               \left[  \  \left\| \frac{\vec{x}}{\|\vec{x}\|} + \frac{\vec{y}}{\|\vec{y}\|}  \right\|^{2} \ - \
               \left\| \frac{\vec{x}}{\|\vec{x}\|} -  \frac{\vec{y}}{\|\vec{y}\|}  \right\|^{2} \  \right]  
               \cdot  \left(    \frac{1}{4} \cdot  { \bf \Sigma}  \right)  ^{\bf \varrho} \, , \ {\rm and} \\ 
               <- \vec{x} \ | \ \vec{y}>_{\bf \varrho}  & = &  
               \frac{1}{4} \cdot  \| -\vec{x}\| \cdot \|\vec{y}\| \cdot 
               \left[  \ \left\| \frac{ -\vec{x}}{\| -\vec{x}\|} + \frac{\vec{y}}{\|\vec{y}\|} \right\|^{2} \  - \
               \left\| \frac{ -\vec{x}}{\| -\vec{x}\|} -  \frac{\vec{y}}{\|\vec{y}\|} \right\|^{2} \  \right] 
               \cdot  \left(    \frac{1}{4} \cdot  { \bf \Sigma}  \right) ^{\bf \varrho}   \\ 
               & = &    \frac{1}{4} \cdot  \|\vec{x}\| \cdot \|\vec{y}\| \cdot 
               \left[  \ \left\| \frac{\vec{y}}{\| \vec{y}\|} - \frac{\vec{x}}{\|\vec{x}\|} \right\|^{2} \  - \
               \left\| \frac{ \vec{x}}{\| \vec{x}\|} +  \frac{\vec{y}}{\|\vec{y}\|} \right\|^{2} \  \right] 
               \cdot  \left(    \frac{1}{4} \cdot  { \bf \Sigma}  \right) ^{\bf \varrho} \ ,    
        \end{eqnarray*}                
             hence \  $ <- \vec{x} \ | \ \vec{y}>_{\bf \varrho} \ = \ - <\vec{x} \ | \ \vec{y}>_{\bf \varrho} $.  
             Then easily follows also for every  real number $ r < 0 $ \ that \  
             $ <r \cdot \vec{x} \ | \ \vec{y}>_{\bf \varrho}  \ = \ r \cdot <\vec{x} \ | \ \vec{y}>_{\bf \varrho} $,
             \ and  the homogeneity $\overline{(1)}$ \  is proven.   
   \end{proof} 
   \begin{definition}    \rm   \label{Definition drei}  Let ${\bf \varrho}$ be a real number.  
	  For  positive definite \ $ \mathsf{BW}$ spaces $ (X , \| \cdot \|)$ \ for two elements \
	   $\vec{x}, \vec{y} \in X \backslash \{\vec{0}\}$ \   with the additional property \ 
	       $ |< \vec{x} \: | \: \vec{y} >_{\bf \varrho} | \ \leq \ \|\vec{x}\| \cdot \|\vec{y}\| $, or equivalently, 
	       $ \left| \frac{1}{4} \cdot  { \bf \Delta}  \right| \cdot   \left(    \frac{1}{4} \cdot 
	        { \bf \Sigma}   \right) ^{\bf \varrho}  \leq 1 $,  \ 
	       we define the `${\bf \varrho}$-angle' $\angle_{\bf \varrho} (\vec{x}, \vec{y})$. \  Let
	 \begin{eqnarray*}      
	         &   \angle_{\bf \varrho} (\vec{x}, \vec{y}) \ := \
	        \arccos{\frac{< \vec{x} \: | \: \vec{y} >_{\bf \varrho}}{ \|\vec{x}\| \cdot \|\vec{y}\| } }  
	        \ =  \    \arccos \left( \frac{1}{4} \cdot  { \bf \Delta} \cdot
               \left( \frac{1}{4} \cdot  { \bf \Sigma}  \right) ^{\bf \varrho} \ \right)   \ =  \  \\
	        &  \arccos{ \left( \frac{1}{4} \cdot   
          \left[  \ \left\| \frac{\vec{x}}{\|\vec{x}\|}  +  \frac{\vec{y}}{\|\vec{y}\|} \right\|^{2} \  - \
          \left\| \frac{\vec{x}}{\|\vec{x}\|}   -   \frac{\vec{y}}{\|\vec{y}\|} \right\|^{2} \  \right] 
          \cdot  \left\langle   \frac{1}{4} \cdot 
                 \left[  \left\| \frac{\vec{x}}{\|\vec{x}\|}   +  \frac{\vec{y}}{\|\vec{y}\|} \right\|^{2}  + 
                 \left\| \frac{\vec{x}}{\|\vec{x}\|} - \frac{\vec{y}}{\|\vec{y}\|} \right\|^{2}   \right]
                                                       \right\rangle  ^{\bf \varrho}  \  \right) }   \ . 
   \end{eqnarray*}               
   \end{definition} 
   \begin{proposition}  \label{proposition eins}  
        Let \  $ \left(  X , < . \: | \: . >_{IP} \right)$  \ be an   ${\mathsf  {IP}}$ space with the inner product
        \ $ < . \: | \: . >_{IP} $ \ and the generated norm  $  \| \cdot \|  $. \
        Then the triple \ $ \left( X , \| \cdot \| ,  < . \: | \: . >_{\varrho}  \right) $ \ 
        fulfils the  { \sf CSB } inequality,  and we have 
        that the   well-known Euclidean angle corresponds to the $\varrho$-angle, i.e. for all \
        $\vec{x},\vec{y} \neq \vec{0}$    \ it holds  \  $ \angle_{\varrho} (\vec{x}, \vec{y}) = 
        \angle_{Euclid} (\vec{x}, \vec{y})$, \   for every real number $\varrho$.
   \end{proposition}
   \begin{proof}
        In Lemma  \ref{lemma 1} it was shown that \ $  < . \: | \: . >_{IP} \ = \ < . \: | \: . >_{\varrho} $, \ 
        for all real numbers $\varrho$. 
   \end{proof}
   \begin{lemma}   \label{lemma  wwwe}
       For positive definite \ $ \mathsf{BW}$ spaces $ (X , \| \cdot \|)$  for any element 
       $\vec{x} \in X , \, \vec{x} \neq \vec{0} $ (i.e. $ \|\vec{x}\| > 0$) the `angles' 
        $\angle_{\bf \varrho} (\vec{x}, \vec{x })$ and $\angle_{\bf \varrho} (\vec{x}, -\vec{x})$ always exist, with  
        $\angle_{\bf \varrho} (\vec{x}, \vec{x }) = 0$ and $\angle_{\bf \varrho} (\vec{x}, -\vec{x}) = \pi$. 
        That means  $({\tt An} \; 2)$ and  $({\tt An} \; 3)$ from Definition \ref{Definition eins}  is
        fulfilled,  for every number  ${\bf \varrho} \in  \mathbbm{R} $.
   \end{lemma}
   \begin{proof}
       Trivial if we use that  $\| \cdot \|$ is balanced and \ $\|\vec{0}\| = 0$. 
   \end{proof}
       Now the reader should take a short look on ({\tt An} 4) - ({\tt An} 7) from Definition \ref{Definition eins}
       to prepare the following proposition.
   \begin{proposition}  \label{proposition fvafav}
       Assume a  positive definite  $ \mathsf{BW}$ space $ (X , \| \cdot \|)$ and  two fixed vectors 
       $\vec{x}, \vec{y} \in X$,  such that the  ${\bf \varrho}$-angle
       $ \angle_{\bf \varrho} (\vec{x}, \vec{y})$ is defined,  for a fixed number $\varrho$. 
       Then the following ${\bf \varrho}$-angles are also defined, and  we have 
     \begin{itemize}
        \item  (a) \  $\angle_{\bf \varrho} (\vec{x}, \vec{y}) = \angle_{\bf \varrho} (\vec{y}, \vec{x}) $. 
        \item  (b) \  $\angle_{\bf \varrho} (r \cdot \vec{x}, s \cdot \vec{y}) = \angle_{\bf \varrho}
          (\vec{x},\vec{y})$, for all positive real numbers  $r,s$. 
        \item  (c) \  $\angle_{\bf \varrho} ( - \vec{x}, - \vec{y}) = \angle_{\bf \varrho} (\vec{x},\vec{y})$.          
        \item  (d) \  $ \angle_{\bf \varrho}(\vec{x}, \vec{y}) + \angle_{\bf \varrho}(-\vec{x}, \vec{y})  = \pi $. 
      \end{itemize}    
   \end{proposition} 
   \begin{proof}
    Easy. We defined $\angle_{\bf \varrho} (\vec{x}, \vec{y}) \ = \
	                  \arccos{\frac{< \vec{x} \: | \: \vec{y} >_{\bf \varrho}}{ \|\vec{x}\| \cdot \|\vec{y}\| } } $,
	        and in Lemma \ref{lemma 2} we proved that the space $ (X, < . \: | \: . >_{\bf \varrho})$ 
	        is a homogeneous product vector  space.  Then (a) is true since  $< . \: | \: . >_{\bf \varrho}$ 
	        is symmetrical.    We have (b) and (c) because  the product is homogeneous, i.e. 
	        $ <  r \cdot \vec{x} \: | \: s \cdot  \vec{y} >_{\bf \varrho}
	        = r \cdot s \, \cdot < \vec{x} \: | \: \vec{y} >_{\bf \varrho} $, for all \ $ r,s \in \mathbbm{R}$ as well as
	        $ \| r \cdot \vec{y}\| = |r| \cdot \|\vec{y}\|$, for real numbers $r$. 
	        And (d) follows because \  $ \arccos{(v)} + \arccos{(-v)} = \pi$, for all $ v $ from the interval $[-1,1]$.       \end{proof}
   We consider three special cases, let us take  $ {\bf \varrho} $ from the set $\{ 1, 0, -1 \}$.  
        For vectors  $\vec{x}, \vec{y} \neq \vec{0}$ of a  positive definite $ \mathsf{BW}$ space $ (X , \| \cdot \|)$ \
           let us assume \     
    $ |< \vec{x} \: | \: \vec{y} >_{\bf 1} | \ \leq \ \|\vec{x}\| \cdot \|\vec{y}\| $  \ or \ 
    $ |< \vec{x} \: | \: \vec{y} >_{\bf 0} | \ \leq \ \|\vec{x}\| \cdot \|\vec{y}\| $  \ or \
    $ |< \vec{x} \: | \: \vec{y} >_{\bf -1} | \ \leq \ \|\vec{x}\| \cdot \|\vec{y}\| $, \ respectively. \ 
     In Definition  \ref{Definition drei} we defined the `${\bf \varrho}$-angle' $\angle_{\bf \varrho}$, 
        including the cases \ $ \angle_{\bf 1} , \,  \angle_{\bf 0}, \,  \angle_{\bf -1}$. \ We have  
  \begin{align*}      
       \angle_{\bf \varrho = 1} (\vec{x}, \vec{y}) \ &
        = \ \arccos{\left( \frac{1}{16} \cdot   
       \left[  \ \left\| \frac{\vec{x}}{\|\vec{x}\|}  +  \frac{\vec{y}}{\|\vec{y}\|} \right\|^{4} \  - \
       \left\| \frac{\vec{x}}{\|\vec{x}\|}   -   \frac{\vec{y}}{\|\vec{y}\|} \right\|^{4} \  \right]  \right)}  
       \  = \  \arccos{\left( \frac{1}{16} \cdot \ \left( { \bf s}^{4} - { \bf d}^{4} \right) \ \right)} ,  \\
       \angle_{\bf \varrho = 0} (\vec{x}, \vec{y}) \ &
        = \ \arccos{\left( \frac{1}{4} \cdot   
       \left[  \ \left\| \frac{\vec{x}}{\|\vec{x}\|}  +  \frac{\vec{y}}{\|\vec{y}\|} \right\|^{2} \  - \
       \left\| \frac{\vec{x}}{\|\vec{x}\|}   -   \frac{\vec{y}}{\|\vec{y}\|} \right\|^{2} \  \right]  \right)}  
       \  =  \ \arccos{\left( \frac{1}{4} \cdot   { \bf \Delta}   \right)}      \ ,                            \\
       \angle_{\bf \varrho = -1} (\vec{x}, \vec{y}) \ &
        = \  \arccos{ \left(   
      \frac{ \left\| \frac{\vec{x}}{\|\vec{x}\|}  +  \frac{\vec{y}}{\|\vec{y}\|} \right\|^{2} \ - \
       \left\| \frac{\vec{x}}{\|\vec{x}\|}   -   \frac{\vec{y}}{\|\vec{y}\|} \right\|^{2} }
           { \left\| \frac{\vec{x}}{\|\vec{x}\|}  +  \frac{\vec{y}}{\|\vec{y}\|} \right\|^{2} \ + \
       \left\| \frac{\vec{x}}{\|\vec{x}\|}   -   \frac{\vec{y}}{\|\vec{y}\|} \right\|^{2} }  \right)} 
       \ = \     \arccos{ \left( \frac{{\bf s}^{2} - {\bf d}^{2}} {{\bf s}^{2} + {\bf d}^{2}}  \right)} 
       \  = \  \arccos \left( \frac{{\bf \Delta}}{{\bf \Sigma}} \right)  \ . 
  \end{align*}  
 	\begin{remark}   \rm 
	      The angle $ \angle_{\bf 0} $ reflects the fact that the cosine of an inner angle in a rhombus with the side 
	      lenght $1$  can be expressed  as the fourth part of the difference of the  squares of the two diagonals, while
	      $ \angle_{\bf -1} $ means   that the cosine of an inner angle in a rhombus with the side 
	      lenght $1$ is the difference of the  squares of the two diagonals divided by its sum. 
	 \end{remark}     
  \begin{proposition}  \qquad   \it    \label{proposition1} 
          For this proposition let  $ (X , \| \cdot \|)$ be a real  normed vector space, and we consider vectors \  
          $ \vec{x}, \vec{y} \in X\backslash\{\vec{0}\}$. \ Let us take $ {\bf \varrho}$ from the set \
          $ \{ 1, 0, -1 \}$. We have \\
	  (a)   The  triple \ $( X , \| \cdot \| , < . \: | \:  . >_{\bf \varrho} ) $ \  \  fulfils the  
	        { \sf CSB } inequality, hence the  \ `$\varrho$- angle' \  $ \angle_{\bf \varrho} (\vec{x}, \vec{y}) $ \
	        is defined for all  \  $ \vec{x}, \vec{y} \neq \vec{0}$. \ It  means that \ the space  $( X , \| \cdot \| )$
	        has the angle $ \angle_{ \bf \varrho } $ \textsl{}, \ for ${\bf \varrho} \in \{ 1, 0, -1 \}$.        \\ 
	  (b)   The  triple \  $(X , \| \cdot \|,\angle_{\bf \varrho})$ \ fulfils  all seven demands 
	        ({\tt An} 1) - ({\tt An} 7). \  Hence \ $ \left( X , \| \cdot \|,\angle_{\bf \varrho} \right) $ is an
	        angle space \ as it was  defined in    { \rm Definition \ref{Definition eins}}. \\ 
	  (c)   The  triple \ $(X , \| \cdot \|,\angle_{\bf \varrho})$ \ generally does not  fulfil \ 
	        ({\tt An} 8), ({\tt An} 9), ({\tt An} 10). \\
	  (d)   In the special case of \  $ {\bf \varrho} = 0 $ \  the  triple \ 
	        $(X , \| \cdot \|,\angle_{\bf 0 })$ \  fulfils \  ({\tt An} 11).       
   \end{proposition}
	  \begin{proof}
	 	  (a) \quad  We show the  { \sf CSB } inequality for  $ {\bf \varrho} = 1 $.
	 	  If \ ($ X , \| \cdot \|$) is a  normed vector space, then because of the triangle inequality 
	           and \ $ \left\| \frac{\vec{x}}{\|\vec{x}\|} \right\| = 1 $  \   we get that 
	 \begin{eqnarray*}          
	           & \left| < \vec{x} \: | \:  \vec{y} >_{\bf 1} \right| \  =  \ 
	           \left|  \frac{1}{16} \cdot  \|\vec{x}\| \cdot \|\vec{y}\| \cdot 
             \left[  \ \left\| \frac{\vec{x}}{\|\vec{x}\|}   +  \frac{\vec{y}}{\|\vec{y}\|} \right\|^{4} \  - \
             \left\| \frac{\vec{x}}{\|\vec{x}\|}   -   \frac{\vec{y}}{\|\vec{y}\|} \right\|^{4} \  \right] \right| \\
	           & \leq   \  \frac{1}{16} \cdot  \|\vec{x}\| \cdot \|\vec{y}\| \cdot 
	           \max \left\{  \left\| \frac{\vec{x}}{\|\vec{x}\|} + \frac{\vec{y}}{\|\vec{y}\|} \right\|^{4},
	           \left\| \frac{\vec{x}}{\|\vec{x}\|} - \frac{\vec{y}}{\|\vec{y}\|} \right\|^{4}  \right\} 
	           \leq   \frac{1}{16} \cdot  \|\vec{x}\| \cdot \|\vec{y}\| \cdot 2^{4}  
	           \ = \  \|\vec{x}\| \cdot \|\vec{y}\| \ . 
   \end{eqnarray*} 
   The same way works with ${\bf \varrho} = 0 $, and for $ {\bf \varrho} = -1 $ the { \sf CSB} inequality is obvious.     \\ (b) \quad  The demands  ({\tt An} 2), ({\tt An} 3) are shown in Lemma \ref{lemma  wwwe}.  The map 
              $ \angle_{\bf \varrho}: [X\backslash\{\vec{0}\}]^{2} \longrightarrow [0,\pi] $ is continuous,  hence 
              ({\tt An} 1) is fulfilled. For ({\tt An} 4), ({\tt An} 5), ({\tt An} 6) and ({\tt An} 7)
              see Proposition \ref{proposition fvafav}. \\        
     (c) \quad We repeat counterexamples from the online publication  { \bf \cite{Thuerey}}. 
             Recall the pairs  ($ \mathbbm{R}^{2} , \| \cdot \|_{p}$),  with 
	           the  { \it H\"older weights } $ \| \cdot \|_{p}$, \ $ p > 0$,  \ we have defined  \ 
	           $ \|(x_1|x_2)\|_{p} := \ \sqrt[p]{|x_1|^{p}+|x_2|^{p}}$. \ The pairs  ($\mathbbm{R}^{2}, \| \cdot \|_{p}$)
	           are normed spaces if and only if \ $ p \geq 1$.   For \ $ p=2$ \ we get the usual Euclidean norm. 
	           Let us take, for instance, \ $ p = 1 $,  because it is easy  to calculate with.  \    
          	 Let \ $ \vec{x} := (1|0), \ \vec{y} := (0|1) $,  both vectors have  a 
          	 H\"older weight $ \| \cdot \|_{1} = 1$ .  \ We choose ${\bf \varrho} := 0$.    Then  we  have
	  \begin{eqnarray*}  
	        \angle_{\bf 0 } (\vec{x}, \vec{y})  
	    &  = &  \arccos{ \left( \frac{1}{4} \cdot   
          \left[  \ \left\| \frac{\vec{x}}{\|\vec{x}\|_1}   +  \frac{\vec{y}}{\|\vec{y}\|_1} \right\|_1^{2} \  - \
          \left\| \frac{\vec{x}}{\|\vec{x}\|_1} - \frac{\vec{y}}{\|\vec{y}\|_1} \right\|_1^{2} \ \right] \right)} \\
      &  =  &  \arccos{ \left( \frac{1}{4} \cdot  \left[  \ \left\| (1|0)  + (0|1) \right\|_1^{2} \  - \
           \left\|  (1|0)  - (0|1) \right\|_1^{2} \  \right] \right)}    \\   
      &  =  &   \arccos{ \left( \frac{1}{4} \cdot  \left[  \ 4 \ - \ 4 \  \right] \right)} \ = \ \arccos (0) \ = 
           \ \pi / 2 \ = \ 90 \deg \  ,   \\
           \angle_{\bf  0} (\vec{x},\vec{x} + \vec{y})  
	    &  =  &  \arccos{ \left( \frac{1}{4} \cdot   
            \left[  \ \left\| (1|0)  +  \frac{1}{2} \cdot  (1|1) \right\|_1^{2} \  - \
            \left\|  (1|0)  - \frac{1}{2} \cdot  (1|1) \right\|_1^{2} \  \right] \right)}    \\   
      &  =  &  \arccos{ \left( \frac{1}{4} \cdot  \left[  \ \left( 2 \right)^{2} \ - \ 
             \left( 1 \right)^{2}  \  \right] \right)}  =  \arccos \left(\frac{3}{4}\right) \ \approx \ 41.41 \deg.
    \end{eqnarray*} 
    With similar calculations, we get \ \ 
           $  \angle_{\bf 0} (\vec{x} + \vec{y},\vec{y})  \ = \  \arccos \left(\frac{3}{4}\right) $,   \quad
           hence \\  $ \ \angle_{\bf 0} (\vec{x},\vec{x} + \vec{y}) \ + \ \angle_{\bf 0} (\vec{x} + \vec{y},\vec{y})  
           \ \neq \ \angle_{\bf 0} (\vec{x},\vec{y})  $,  \  and    that contradicts \ ({\tt An} 8).  
                            
   The property \ ({\tt An} 9) \ means that the sum of the inner angles of a triangle is $\pi$.    \\
        We can use the same example of the normed space \  ($ \mathbbm{R}^{2} , \| \cdot \|_{1}$) \ with  unit vectors
        $ \vec{x} = (1|0)$,  and \ $\vec{y} = (0|1) $. Again we get \ \ 
        $\angle_{\bf 0} (\vec{x},\vec{y}) = \pi/2 ,  \ \ 
        \angle_{\bf 0} (-\vec{x},\vec{y}-\vec{x}) =  \angle_{\bf 0 } (-\vec{y},\vec{x}-\vec{y}) 
                                                                =  \arccos\left(\frac{3}{4}\right) $,  \  hence \
        $ \angle_{\bf 0 } (\vec{x},\vec{y}) +  \angle_{\bf 0} (-\vec{x},\vec{y}-\vec{x}) +  \angle_{\bf 0 }
             (-\vec{y},\vec{x}-\vec{y}) < \pi $,  hence \ ({\tt An} 9) \ is not fulfilled. 
                
     For the condition ({\tt An} 10)  we use the same space and the same vectors 
             $ \vec{x} = (1|0)$,  and \ $\vec{y} = (0|1)$. \ We get  $\angle_{\bf 0} (-\vec{x},\vec{y}) = \pi/2, 
             \ \angle_{\bf 0} (\vec{y},\vec{y}-\vec{x}) =  \angle_{\bf 0} (\vec{x},\vec{x}-\vec{y}) 
             =  \arccos\left(\frac{3}{4}\right) $, \ hence  ({\tt An} 10)  is not fulfilled. \\
     (d)  \quad This was shown in  { \bf \cite{Thuerey}} on `ArXiv'.    \\    
      Now the proof of the proposition is complete.            
	\end{proof}  
	\begin{remark}   \rm
	     Note that one  $ {\bf \varrho}$-angle  was considered first by Pavle M. $ \rm  Mili\check{c}i\acute{c}$,
	     see the references  { \bf \cite{Milicic1}}, { \bf \cite{Milicic2}},  { \bf \cite{Milicic3}}, where he dealt 
	     with the   case $ \bf \varrho = 1 $.  He named his angle as the  `{\sf g}-angle'.
	     In the recent paper { \bf \cite{Milicic4}} it is shown that the different definitions of 
	     the angle $\angle_{ \bf 1 }$   and the  `{\sf g}-angle' are equivalent  at least in 
	     {\it quasi-inner-product spaces}.  The  case   $ \bf \varrho = 0 $ was introduced by the author in 
	     { \bf \cite{Thuerey}}. There it was  called the `{\sf Thy}-angle'. In  { \bf \cite{Milicic4}} some properties
	     of the {\sf g}-angle and the  { \sf Thy}-angle are compared.
	\end{remark}  
	\newpage
	\section{On Classes and Corners} 
	    Now we define some classes of real  $\mathsf{BW}$ spaces  	and  real normed spaces. 
	\begin{definition}  \rm  \label{Definition  vier} $ $  
	    Let  \  $ \mathsf{pdBW}$ \ be the class of all real positive definite \ $ \mathsf{BW}$ spaces. \\
    	Let \ \  $ \mathsf{NORM}$ \ be the class of all real normed vector spaces. \\
    	Let \ \  $ \mathsf{IPspace}$ \ be the class of all real \ inner product spaces (or ${\mathsf {IP}}$ spaces).   \\ 
    	For a fixed real number ${\bf \varrho}$ let
    	\begin{eqnarray*} 
    	      \mathsf{pdBW}_{\bf \varrho} \  := \ &  \{ (X , \| \cdot \|) \in \mathsf{pdBW} \ | \ 
	         (X , \| \cdot \|) \text{ has the angle }  \angle_{\bf \varrho} \}  \\
            \mathsf{NORM}_{\bf \varrho} \  := \ & \{ (X , \| \cdot \|) \in \mathsf{NORM} \ | \ 
	         \text{ The normed space } (X , \| \cdot \|) \text{ has the angle }  \angle_{\bf \varrho} \}. 
	   \end{eqnarray*}      
	          \hfill  $\Box$   
  \end{definition}     
   We have \quad   
          $ \mathsf{IPspace} \subset \mathsf{NORM} \subset \mathsf{pdBW} \subset \mathsf{BW} \ \text{spaces} $ 
          \ \ and \ \ 
          $ \mathsf{NORM}_{\bf \varrho} \subset \mathsf{pdBW}_{\bf \varrho}$, \  of course.  
  \begin{proposition}
  For all real numbers $ \bf \varrho $ it holds that every element  $ (X , \| \cdot \|) $ of  
          $ \mathsf{pdBW}_{\bf \varrho} $ is an angle space as it was defined in Definition \ref{Definition eins}.
  \end{proposition}  
  \begin{proof}  By definition of the class  $ \mathsf{pdBW}_{\bf \varrho} $  each element
           $(X , \| \cdot \|) \text{ has the angle }  \angle_{\bf \varrho} $.  Further, by definition of the angle 
           $ \angle_{\bf \varrho} $ all seven properties  ({\tt An} 1) -  ({\tt An} 7)  of Definition 
           \ref{Definition eins} are fulfilled.   
  \end{proof}  
	\begin{proposition} \  \label{corollary eins}
	       It holds \ \  $ \mathsf{pdBW}  = \mathsf{pdBW}_{\bf -1}$ \ \ and \ \ 
	       $ \mathsf{NORM}  = \mathsf{NORM}_{\bf -1} = \mathsf{NORM}_{\bf 0} = \mathsf{NORM}_{\bf 1} $.  
	\end{proposition}
	\begin{proof}    \quad We had defined \
	       $  \angle_{\bf -1} (\vec{x}, \vec{y}) \ = \  
         \arccos{ \left( \frac{{\bf s}^{2} - {\bf d}^{2}} {{\bf s}^{2} + {\bf d}^{2}}  \right)} $.  
         By this definition  this angle always exists for all $ \vec{x}, \vec{y} \neq \vec{0}$. \ 
         For the second claim see Proposition \ref{proposition1}. 
	\end{proof} 
	\begin{theorem} \label{erstes theorem}
	     Let us take four real  numbers \ $\alpha, \beta, \gamma , \delta$ \ with \\
	                \centerline{   $ -\delta < -\gamma < -1 <  \alpha <  \beta $.} 
      	$$  \text{ We get the inclusions }  \quad 
	                \mathsf{pdBW}_{\bf -\delta} \subset   \mathsf{pdBW}_{\bf -\gamma} \subset    \mathsf{pdBW}
	                \supset  \mathsf{pdBW}_{\bf \alpha} \supset  \mathsf{pdBW}_{\bf \beta} \ .   $$
  \end{theorem}
	\begin{proof} \quad First we consider $  -1 <  \alpha <  \beta $. 
	         
	   Let $(X , \| \cdot \|) \in \mathsf{pdBW}_{\bf \beta}$.  By Definition \ref{Definition  vier},  for each pair 
	   of two vectors   $ \vec{x}, \vec{y} \neq \vec{0}$  the angle  $\angle_{\bf \beta} (\vec{x}, \vec{y})$ is defined.
	   \ By Definition  \ref{Definition drei}, this means that  the triple 
	  \ $ \left( X, \| \cdot \| ,  < . \: | \:  . >_{\bf \beta} \right) $ fulfils the {\sf CSB }inequality, \
	  i.e. for any pair $ \vec{x}, \vec{y} \neq \vec{0}$ \ of vectors  we have the inequality  
	\begin{align}  
	    \left| < \vec{x} \: | \:  \vec{y} >_{\bf \beta} \right|  \leq \ \|\vec{x}\| \cdot \|\vec{y}\| \ ,
	       \quad  \text{that means}   \quad 
	  \left| \ \|\vec{x}\| \cdot \|\vec{y}\| \cdot \frac{1}{4} \cdot  { \bf \Delta } \cdot
     \left(  \frac{1}{4} \cdot  { \bf \Sigma}   \right) ^{\bf \beta} \ \right|
     &  \leq \|\vec{x}\| \cdot \|\vec{y}\| \ ,  \\
      \text{ or equivalently}  \quad   \left| \frac{1}{4} \cdot  { \bf \Delta} \right| \cdot 
	        \left( \frac{1}{4} \cdot {\bf \Sigma} \right) ^{\bf \beta} \ \leq 1 \ .  
  \end{align}           
	    To prove that the angle 
  	  $\angle_{\bf \alpha} (\vec{x}, \vec{y})$ exists  we have to show the corresponding  inequality  
	  $$ \left| \frac{1}{4} \cdot  { \bf \Delta}  \right| \cdot   \left( \frac{1}{4} \cdot 
	                  { \bf \Sigma}   \right)  ^{\bf \alpha}  \leq 1   \, .  $$
     We distinguish two cases. In the first case of \  $  \frac{1}{4} \cdot {\bf \Sigma} \geq 1 $  we have 
     for all real numbers \  $ { \bf \kappa} \leq {\bf \beta}$
	   $$   \left(    \frac{1}{4} \cdot {\bf \Sigma} \right\rangle ^{\bf \kappa} 
	          \leq   \left(    \frac{1}{4} \cdot {\bf \Sigma}   \right) ^{\bf \beta}  . $$
	   Since $ \alpha < \beta $ \ it follows that 
	   $$ 0 \leq  \left| \frac{1}{4} \cdot  { \bf \Delta}  \right| \cdot 
	      \left( \frac{1}{4} \cdot { \bf \Sigma}   \right) ^{\bf \alpha} \leq 
	      \left| \frac{1}{4} \cdot { \bf \Delta}\right| \cdot 
	      \left( \frac{1}{4} \cdot {\bf \Sigma} \right) ^{\bf \beta} \leq 1 ,   $$  
	  and the angle $ \angle_{\bf \alpha} (\vec{x}, \vec{y})$ exists.   
	                           
	  For the second case we assume \ $ \frac{1}{4} \cdot {\bf \Sigma} < 1. $  That means for any positive
	  exponent $ \kappa$ \  $ \left(    \frac{1}{4} \cdot {\bf \Sigma} \right) ^{\bf \kappa} < 1 $. \
	   Now note  the inequality \  
	   $ 0 \leq  \left| \frac{1}{4} \cdot  { \bf \Delta}  \right| \leq  \frac{1}{4} \cdot {\bf \Sigma} < 1 $.  \
	   In the subcase of a positive $ \alpha $ it follows the inequality \  
	   $ \left| \frac{1}{4} \cdot  { \bf \Delta}  \right| \cdot 
	   \left( \frac{1}{4} \cdot  { \bf \Sigma}   \right) ^{\bf \alpha} < 1 $. \
	   and the angle $ \angle_{\bf \alpha} (\vec{x}, \vec{y})$ exists.  
	         
	 If $ \alpha $ is from the interval \ $[-1,0]$, \ i.e. \ $ -\alpha  \in [0,1]$, we can write the inequality
	  $$ 1 \geq \frac{\left| {\bf \Delta}  \right|}{ {\bf \Sigma}} = 
	     \frac{\left| \frac{1}{4} \cdot {\bf \Delta}  \right|}{ \frac{1}{4} \cdot { \bf \Sigma}} 
	     \geq \frac{\left| \frac{1}{4} \cdot {\bf \Delta}  \right|}
	      { \left( \frac{1}{4} \cdot { \bf \Sigma} \right)^{-\alpha} } 
	      \geq \left| \frac{1}{4} \cdot {\bf \Delta} \right| \ . $$    
	   Again we get the desired inequality \  
	    $ \left| \frac{1}{4} \cdot {\bf \Delta}  \right| \cdot \left( \frac{1}{4} \cdot 
	        { \bf \Sigma} \right)  ^{\bf \alpha}  \leq 1 $, \
	  and the angle $ \angle_{\bf \alpha} (\vec{x}, \vec{y})$ exists. \  We get that 
	  $(X , \| \cdot \|)$ is an element of $ \mathsf{pdBW}_{\bf \alpha}$, too.  \\   
	                             
	   We look at $   -\delta <  -\gamma < -1 $. \ We have \ $  1 < \gamma < \delta $.  \\
	    Let $(X , \| \cdot \|) \in \mathsf{pdBW}_{\bf -\delta}$, and  take  two vectors 
	    $\vec{x}, \vec{y} \in X, \ \vec{x}, \vec{y} \neq \vec{0} $.
	    The angle $ \angle_{\bf -\delta} (\vec{x}, \vec{y})$ exists.  Hence we have the inequality \      
	    $ \left| \frac{1}{4} \cdot  { \bf \Delta}  \right| \cdot 
	    \left(  \frac{1}{4} \cdot { \bf \Sigma} \right) ^{\bf -\delta}  \leq 1 $.  \  
	           
	  As above we distinguish two cases.  The first case is \ $\frac{1}{4} \cdot {\bf \Sigma} \geq 1 $. We have                  $$ 1 \geq \frac{\left| {\bf \Delta}  \right|}{ {\bf \Sigma}} = 
	     \frac{\left| \frac{1}{4} \cdot {\bf \Delta}  \right|}{ \frac{1}{4} \cdot { \bf \Sigma}} 
	     \geq \frac{\left| \frac{1}{4} \cdot {\bf \Delta}  \right|}
	      { \left( \frac{1}{4} \cdot { \bf \Sigma} \right)^{\gamma} } 
	      \geq \frac{\left| \frac{1}{4} \cdot {\bf \Delta}  \right|}
	      { \left( \frac{1}{4} \cdot { \bf \Sigma} \right)^{\delta} } 
	      =  \left| \frac{1}{4} \cdot  { \bf \Delta}  \right| \cdot 
	      \left( \frac{1}{4} \cdot { \bf \Sigma} \right) ^{\bf -\delta} \ . $$     
	    We get the inequality \   $ \left| \frac{1}{4} \cdot  { \bf \Delta}  \right| \cdot 
	    \left(  \frac{1}{4} \cdot { \bf \Sigma} \right) ^{\bf -\gamma}  \leq 1 $.  \   
	     It follows that  the angle $ \angle_{\bf -\gamma} (\vec{x}, \vec{y})$ exists.  \\
	      The second case is \ $\frac{1}{4} \cdot {\bf \Sigma} < 1 $. We get 
	     $$ 0 \leq  \left( \frac{1}{4} \cdot { \bf \Sigma} \right) ^{\bf \delta} 
	          \leq  \left( \frac{1}{4} \cdot { \bf \Sigma} \right) ^{\bf \gamma} 
	          \leq     \frac{1}{4} \cdot { \bf \Sigma}      \ , \ \text{ hence } \  
	          \frac{\left| \frac{1}{4} \cdot {\bf \Delta}  \right|}
	           { \left( \frac{1}{4} \cdot { \bf \Sigma} \right)^{\gamma} } 
	          \leq    \frac{\left| \frac{1}{4} \cdot {\bf \Delta}  \right|}
	            { \left( \frac{1}{4} \cdot { \bf \Sigma} \right) ^{\delta} } \leq 1 \ ,          $$  
	  and the angle $ \angle_{\bf -\gamma} (\vec{x}, \vec{y})$ exists.  
	  This proves  $(X , \| \cdot \|) \in \mathsf{pdBW}_{\bf -\gamma}$, and Theorem \ref{erstes theorem} is shown. 
 	\end{proof}
	\begin{corollary}  
	  	$$  \text{ We have} \ \ \mathsf{NORM}  =  \mathsf{NORM}_{\bf \varrho} \ \ \text{for all real numbers} 
	                    \ {\bf \varrho} \ \text{from the closed interval } \ [-1 , 1 ] \ .   $$
	\end{corollary}
  \begin{proof} \quad See both the above Proposition  \ref{corollary eins}  and  Theorem \ref{erstes theorem}.  
	\end{proof}  
	\begin{corollary} \label{zweites Korollar}
	Let us take four positive numbers \ $\alpha, \beta, \gamma , \delta $ \	with \\ 
	                \centerline{   $ -\delta < -\gamma < -1 <  1 < \alpha <  \beta $.} 
	$$  \text{ We have }  \qquad 
	                \mathsf{NORM}_{\bf -\delta} \subset   \mathsf{NORM}_{\bf -\gamma} \subset    \mathsf{NORM}
	                \supset  \mathsf{NORM}_{\bf \alpha} \supset  \mathsf{NORM}_{\bf \beta} \ .   $$                
	\end{corollary}
	\begin{proof}  \quad This follows directly from  Theorem \ref{erstes theorem}. 
  \end{proof}
 	\begin{theorem}   \label{zweites theorem}       
	             We have the equality 
	             $$ \mathsf{IPspace} \ = \ \bigcap_{{\bf \varrho} \in \mathbbm{R}} \ \mathsf{NORM}_{\bf \varrho} \ . $$
	\end{theorem} 
	\begin{proof} \quad      $" \subset "$: This is trivial with Proposition \ref{proposition eins}.   \\
	   $ \ {  }  $  \qquad \qquad    $" \supset "$: This is not trivial, but easy.  We show that a real normed space  
	    $ ( X ,  \| \cdot \| )$ which in not an inner product space is not an element of   
	   $  \mathsf{NORM}_{\bf \varrho}$ for at least one real number \ $\bf \varrho$. \ 
	             
	    Let   $ ( X ,  \| \cdot \| )$ be a real normed space which in not an inner product space. Then it  
	    must exist a  two dimensional subspace $ \mathsf{U} $ of $ X $ such that its unit sphere \ 
	    $ {\bf S} \cap \mathsf{U} $ \  is not an ellipse.  \ Hence there are two unit vectors \ 
	    $ \vec{v} , \vec{w} \in  {\bf S} \cap \mathsf{U} $ such that the parallelegram identity is not fulfilled; 
	    i.e. it holds  
	 $$ \|\vec{v}+\vec{w}\|^{2} + \|\vec{v}-\vec{w}\|^{2} \neq 4
	                                = 2 \cdot \left[\|\vec{v}\|^{2} + \|\vec{w}\|^{2} \right] \ . $$
	                                                  
	 \underline{(Case A)}: 
	 First we assume \ $\|\vec{v}+\vec{w}\| \neq \|\vec{v}-\vec{w}\|$, \ hence \ 
	    $ { \bf \Delta} :=  { \bf \Delta} (\vec{v},\vec{w}) \neq 0$. 
	    In the case of  $ \|\vec{v}+\vec{w}\|^{2} + \|\vec{v}-\vec{w}\|^{2} > 4 $, i.e. \
	    $ \frac{1}{4} \cdot {\bf \Sigma} > 1$,  we can choose a very big number   $ \beta $ \ such that 
	       $$ \left| \ \|\vec{v}\| \cdot \|\vec{w}\| \cdot \frac{1}{4} \cdot  { \bf \Delta} \cdot
     \left\langle  \frac{1}{4} \cdot {\bf \Sigma} \right\rangle  ^{\bf \beta} \ \right|
      > \|\vec{v}\| \cdot \|\vec{w}\| = 1 \ ,   $$     
      and if \  $ |\vec{v}+\vec{w}\|^{2} + \|\vec{v}-\vec{w}\|^{2} < 4 $, i.e. $ \frac{1}{4} \cdot {\bf \Sigma} < 1$,
      we can find a big $\gamma$ such that
      $$ \left| \ \|\vec{v}\| \cdot \|\vec{w}\| \cdot \frac{1}{4} \cdot  { \bf \Delta} \cdot
     \left\langle  \frac{1}{4} \cdot {\bf \Sigma} \right\rangle  ^{-\bf \gamma} \ \right|
      > \|\vec{v}\| \cdot \|\vec{w}\| = 1 \ .   $$ 
      We get that the angle \ $\angle_{\bf \beta} (\vec{v}, \vec{w})$ \ or \
      $ \angle_{\bf -\gamma} (\vec{v}, \vec{w})$, respectively,  does not exist.  
                
   \underline{(Case B)}: 
   If  we have \ $\|\vec{v}+\vec{w}\| = \|\vec{v}-\vec{w}\|$, \ hence \ $ { \bf \Delta} = 0$,
      we have to replace  \ $ \vec{w} $ \ by another unit vector  \ $ \widetilde{w}$. Note that 
      $ \{ \vec{v}, \vec{w} \} $ 
      is linear independent since $ \|\vec{v}+\vec{w}\|^{2} + \|\vec{v}-\vec{w}\|^{2} \neq 4$. \
      We regard  the continuous map  \ $ \underline{E} :  \mathbbm{R} \longrightarrow  (-1,+1)  $, \   we define 
      $$  \underline{E} (t) \ := \
                         \frac{1}{4} \cdot  \left[  \ \left\| \vec{v} + \frac{\vec{w} + t \cdot \vec{v}}{\|\vec{w} + 
                         t \cdot \vec{v}\|} \right\|^{2} \ - \ \left\| \vec{v}  -  \frac{\vec{w} + 
                         t \cdot \vec{v}}{\|\vec{w} + t \cdot \vec{v}\|} \right\|^{2} \  \right]  \ .   $$                    For \ $t =0$ \ we  get \ 
           $  \underline{E}(0) = \frac{1}{4} \cdot \left[ \ \left\| \vec{v}  +  \vec{w} \right\|^{2} \ - 
           \ \left\| \vec{v} -  \vec{w} \right\|^{2} \  \right]  =  \frac{1}{4} \cdot  { \bf \Delta} = 0 $ .   \ \            In   { \bf \cite{Thuerey}} \ on the internet platform `arXiv' it is proven that the map   $ \underline{E} $ 
           yields a homeomorphism from $  \mathbbm{R} $ onto the open interval $ ( -1, 1 ) $. \        
      Hence  we can replace the factor \ $t=0$ \ by any  $ \widetilde{t} \neq 0$ such that 
  $$    \underline{E} (\widetilde{t}) \ = \  \frac{1}{4} \cdot \left[  \ \left\| \vec{v}  +  
        \frac{\vec{w} + \widetilde{t} \cdot \vec{v}}{\|\vec{w} +   \widetilde{t} \cdot \vec{v}\|} \right\|^{2} \ - \
        \left\| \vec{v}  -  \frac{\vec{w} +  \widetilde{t} \cdot \vec{v}}{\|\vec{w} + 
        \widetilde{t} \cdot \vec{v}\|} \right\|^{2} \ \right] \ \neq \ 0 \ . $$    
      For each  \ $ \widetilde{t}$ \ we abbreviate   the unit vector 
  $$  \widetilde{w} := \frac{\vec{w} + \widetilde{t} \cdot \vec{v}}{\|\vec{w} + \widetilde{t} \cdot \vec{v}\|} \ , $$ 
      and since  $ \underline{E} $ is a homeomorphism   we can choose a very small  $ \widetilde{t} \neq 0$ such
      that   still holds  
      $ \|  \vec{v} +  \widetilde{w} \|^{2} + \|  \vec{v} -  \widetilde{w} \|^{2} \neq 4 $, but
      $ { \bf \Delta} :=  { \bf \Delta} (  \vec{v} ,  \widetilde{w} ) \neq 0$.  \
      At this point we can continue as in (Case A).   
                
	 In both cases (Case A)  and (Case B) it follows that \  $ ( X ,  \| \cdot \| )$
	     is not an element of the classes \ $ \mathsf{NORM}_{\bf \beta}$ or $ \mathsf{NORM}_{-\bf \gamma}$, respectively.
	     Now the proof of  Theorem \ref{zweites theorem} is finished.  
	\end{proof}
      $ $        \\  \\  
     We  define a function  $\Upsilon $  which maps every real positive definite   $ \mathsf{BW}$ space
     to a pair of extended numbers  $ ( \nu , \mu )$, \\
     \centerline {
      $ \Upsilon:  \mathsf{pdBW} \longrightarrow \left[ -\infty, -1 \right] \times \left[ -1 , +\infty \right] $. } \\
   \begin{definition}   \rm   \label{definition fuenf}
        Let  $ ( X ,  \| \cdot \| )$ \ be a positive definite balancedly weighted vector space. We define 
     \begin{align*}         
          \nu & \ := \ \inf \{ \kappa \in  \mathbbm{R} \ | \  ( X ,  \| \cdot \| ) \text 
       { has the angle }  \angle_{\bf \kappa}\}, \\ 
         \mu  & \ := \ \sup \{ \kappa \in  \mathbbm{R} \ | \  ( X ,  \| \cdot \| ) \text 
       { has the angle }  \angle_{\bf \kappa}\}, \\ 
        \Upsilon( X ,  \| \cdot \| ) & \ := \ ( \nu, \mu ) .
     \end{align*} 
   \end{definition} 
       With Theorem  \ref{erstes theorem} we get that  $ \nu  $ is from the interval \ 
       $ \left[ -\infty , -1 \right] $ and  $ \mu $ is from the interval \ $  \left[ -1 , +\infty \right] $.  \
       If  $ ( X ,  \| \cdot \| )$   is even a normed vector space we have \
       $  \mu \in  \left[ +1, +\infty \right] $.  \ 
       If $ ( X ,  \| \cdot \| )$   is even an  inner product space it follows from  Theorem  \ref{zweites theorem}
       the identity \  $  \Upsilon( X ,  \| \cdot \| )  = ( -\infty , + \infty ) $.
  \begin{proposition} 
         Let  $ ( X ,  \| \cdot \| ) \in  \mathsf{pdBW} $, i.e. \ $( X ,  \| \cdot \| )$  is a positive definite
         balancedly  weighted vector space. \  We defined \  $\Upsilon( X ,  \| \cdot \| ) = ( \nu, \mu )$. 
         Let us  assume $\nu  \neq -\infty$ and $ \mu \neq \infty $.   \ We claim that
         the infimum and the supremum will be attained, i.e. we claim 
     \begin{align*}    
          \nu =  \min \{ \kappa \in  \mathbbm{R} \ | \  ( X ,  \| \cdot \| ) \text
         { has the angle }  \angle_{\bf \kappa}\}, \ \
          \mu = \max \{ \kappa \in  \mathbbm{R} \ | \  ( X ,  \| \cdot \| ) \text 
         { has the angle }  \angle_{\bf \kappa}\}.  
      \end{align*}    
  \end{proposition}   
  \begin{proof}
      We show the first claim \\
      \centerline{     $ \nu = \min \{ \kappa \in  \mathbbm{R} \ | \  ( X ,  \| \cdot \| ) \text { has the angle } 
                                                      \angle_{\bf \kappa} \}$.           }   \\
       Let us assume the opposite, i.e. we assume that the angle   $\angle_{\bf \nu}$ does not exist in
       $( X , \| \cdot \| )$. 
       This means $ \nu < -1 $, since \ $ \mathsf{pdBW}  = \mathsf{pdBW}_{\bf -1}$. \
       Hence there are two unit vectors $\vec{v}, \vec{w} \in X $ such that 
    \begin{equation}   \label{equation  soundsoviel}
           |< \vec{v} \: | \:  \vec{w} >_{\bf \nu}|  \ \ = \  
           \frac{1}{4} \cdot | { \bf \Delta}(\vec{v}, \vec{w})| \cdot \left\langle  \frac{1}{4} \cdot 
           { \bf \Sigma}(\vec{v}, \vec{w})  \right\rangle  ^{\bf \nu}  \ = \ 1 + \varepsilon  \ \ \ 
           \text{  for a positive } \  \varepsilon   \ .                             
    \end{equation}       
     We make  the exponent $\nu$ `less negative'. Since \   $\nu < -1$ and
     $ 0 \leq  | { \bf \Delta}| \leq { \bf \Sigma} $ \ it has to  be $ { \bf \Sigma} < 4 $. \ 
     Since $ \mathsf{pdBW}  = \mathsf{pdBW}_{\bf -1}$  we have $ |< \vec{v} \: | \:  \vec{w} >_{-1}| \leq 1  $.  \   
     By the continuity of the left hand side  of the above Equation  (\ref{equation  soundsoviel})
     we can find two positive numbers \
     $ \overline{\eta}, \overline{\lambda} $ \ with \ $ \nu  < \nu + \overline{\eta} < -1 $ \ and \ 
     $ 0 < \overline{\lambda} < \varepsilon $ \ such that 
    \begin{equation}   \label{equation  soundsovielundeins} 
            1 \ < \  |< \vec{v} \: | \:  \vec{w} >_{\bf \nu + \overline{\eta}}|  \ \ = \  \frac{1}{4} \cdot | 
            { \bf \Delta}(\vec{v}, \vec{w})| \cdot \left\langle  \frac{1}{4} \cdot { \bf \Sigma}(\vec{v}, \vec{w}) 
            \right\rangle  ^{\bf \nu + \overline{\eta}}  \ = \ 1 + \overline{\lambda} < \ 1 + \varepsilon  \; . 
     \end{equation}        
      By the continuity of the product \ $ < . \: | \:  . > $ \ we can choose the positive number 
      $\overline{\eta}$  such that         
  $$  1 \ < \ 1 + \overline{\lambda} \ \leq \ 1 + \lambda \ = \ 
      |< \vec{v} \: | \:  \vec{w} >_{\bf \nu + \eta}| \ \leq \ \ 1 + \varepsilon  \;    $$
      holds for all $ \eta$, for \ $ 0 \leq  \eta \leq  \overline{\eta} $ \ for positive  $ \lambda $ with 
      $  \overline{\lambda} \leq \lambda \leq  \varepsilon  $.  \ 
      We conclude that for all \ $ 0 \leq  \eta \leq  \overline{\eta} $ \
      the angle  \ $\angle_{\bf \nu + \eta}(\vec{v}, \vec{w})$ \ does not exist in \ $( X , \| \cdot \| )$. 
      This contradicts the 
      definition of \ $ \nu $ \ as an infimum. This proves the first claim of the proposition.                  
  \end{proof}                                            
                                     
   For the next proposition we need the term of a { \it `strictly convex normed space'}. 
     
  \begin{definition}    \rm   \label{stricly convex}
    A $ \mathsf{BW}$ space  $ ( X ,  \| \cdot \| )$  is called {  \it `strictly convex' } if and only if 
       the interior of the line \  $ \{ t \cdot \vec{u} + (1-t) \cdot \vec{v} \ | \ 0 \leq t \leq 1 \} $ 
       lies in the interior of  the unit ball of $ ( X ,  \| \cdot \| )$ for each pair of distinct unit vectors 
       $( \vec{u}, \vec{v}) , \ \vec{u} \neq \vec{v} $.  That means it holds  \\
    \centerline{   $  \| t \cdot \vec{u} + (1-t) \cdot \vec{v} \| < 1 $ \ for \ $ 0 < t < 1 $. }
  \end{definition}   
  \begin{definition}      \rm        \label{stricly curved}
   We call a $ \mathsf{BW}$ space   $ ( X ,  \| \cdot \| )$  { \it `strictly curved' } if and only if 
       for each pair of distinct unit vectors \ $( \vec{u}, \vec{v}) , \ \vec{u} \neq \vec{v} $ \ 
       the line \  $ \{ t \cdot \vec{u} + (1-t) \cdot \vec{v} \ | \ 0 \leq t \leq 1 \} $ \
       contains at least one element \ $ \widehat{t}$ \ with \ $ 0 < \widehat{t} <  1 $ \ 
       which is not a unit vector, i.e. for \ $ \widehat{t} $ \ holds  \\
    \centerline{   $  \left\| \widehat{t} \cdot \vec{u} + (1-\widehat{t}) \cdot \vec{v} \right\| \neq 1 $. }
  \end{definition}   
        Note that in normed spaces both definitions are equivalent. Further, a positive definite  $ \mathsf{BW}$ space 
        which is strictly convex is a normed space, and it is strictly curved. 
    
        Further, a $ \mathsf{BW}$ space 
        $( X , \| \cdot \| )$ which is not strictly curved must contain a piece of a straight line
        which is completely  in the unit sphere of $ ( X ,  \| \cdot \| )$. As examples we can take the two 
        H\"older norms \ $ \| \cdot \|_{1} $ \ and \ $ \| \cdot \|_{\infty} $ on $\mathbbm{R}^{2}$. \
        The unit spheres of both spaces have the shape of a square.  \ The   H\"older weights \ $ \| \cdot \|_{p} $ \
        on $\mathbbm{R}^{2}$ with \ $ 0 < p < 1 $ \ yield examples of $ \mathsf{BW}$ spaces which are 
        strictly curved, but not strictly convex. 
  \newpage
   \begin{proposition}      \label{proposition  fuenf oder sechs}
         Let \ $ ( X ,  \| \cdot \| )$ \ be a real positive definite $ \mathsf{BW}$ space. 
         Let $  \Upsilon( X ,  \| \cdot \| ) = ( \nu, \mu ) $. \  
         If   $ ( X ,  \| \cdot \| )$  is not strictly curved  we have \ $ \mu = 1 $.
   \end{proposition}
   \begin{proof}  Let us consider a $ \mathsf{BW}$ space $ ( X , \| \cdot \| )$ which is not strictly curved. 
          As we said above it  contains a piece of a straight line  which is completely  in the unit sphere. 
          This fact described in formulas   means that we have two unit vectors $ \vec{z} , \vec{w} $  and a 
          positive number $ 0 < {\bf r} < 1 $    such that 
    $$  \| \vec{z} +  t \cdot \vec{w} \| = 1 \quad \text{ holds for all } \ t \in [- {\bf r},  {\bf r}] \ . $$  
          Now we show that for each exponent $ {\bf \varrho} > 1 $ we can find two unit vectors $\vec{x}, \vec{y}$
          with the property   $ \left| \frac{1}{4} \cdot  { \bf \Delta}  \right| \cdot   \left\langle 
          \frac{1}{4} \cdot   { \bf \Sigma}   \right\rangle  ^{\bf \varrho} > 1 $.  That means that the 
          ${\bf \varrho}$-angle $ \angle_{\bf \varrho}(\vec{x}, \vec{y}) $ does not exist.   
                                                         
      Let us take the unit vectors \ $ \vec{x} := \vec{z} +  t \cdot \vec{w} \text{ and }                                          \vec{y} := \vec{z} -  t \cdot \vec{w}$ \ for   $ 0 <  t < {\bf r} \ . $
           With \ ${ \bf \Delta} = { \bf \Delta}(\vec{x},\vec{y})$ \ and \
           ${ \bf \Sigma} = {\bf \Sigma}(\vec{x},\vec{y})$  \ we consider the desired inequality \
             $ \left| \frac{1}{4} \cdot  { \bf \Delta}  \right| \cdot   \left\langle 
          \frac{1}{4} \cdot   { \bf \Sigma}   \right\rangle  ^{\bf \varrho} > 1 $, i.e.
     \begin{equation}     \label{gleichung sieben}
          \left| \frac{1}{4} \cdot  { \bf \Delta}  \right| \cdot
          \left( \frac{1}{4} \cdot   { \bf \Sigma}   \right)  ^{\bf \varrho} \
          = \ \left| \frac{1}{4} \cdot  \left({ \bf s}^{2} - { \bf d}^{2}\right)  \right| \cdot 
          \left( \frac{1}{4} \cdot  \left({ \bf s}^{2} + { \bf d}^{2}\right)  \right)  ^{\bf \varrho} \
          = \ \left(1-t^{2}\right) \cdot \left(1+ t^{2}\right)^{\bf \varrho}    \  > \ 1 \ .   
     \end{equation}      
          This is equivalent to
     \begin{align}   \label{gleichung acht}
          {\bf \varrho} \ > \ - \ \frac{\log\left(1-t^{2}\right)}{\log\left(1+t^{2}\right)} \ .   
     \end{align}          
          The right hand side is greater than $1$ for all  $ 0 <  t < {\bf r} $.  
          By the rules of L'Hospital we get the limit \ 
      $$ \lim_{ t \searrow 0} \left( - \ \frac{\log\left(1-t^{2}\right)}{\log\left(1+t^{2}\right)}\right) \ = \ 1 \ . $$
           This means that we can find for all  $ {\bf \varrho} > 1 $ a suitable \ $ t $ \ such that
           Inequality (\ref{gleichung sieben}) is fulfilled. \
           Hence, for each  ${\bf \varrho} > 1$, \ we are able to find a pair of unit vectors \
           $\vec{x} = \vec{z} +  t \cdot \vec{w}$ \ and \ $ \vec{y} = \vec{z} -  t \cdot \vec{w} $ \ 
           such that the ${\bf \varrho}$-angle $ \angle_{\bf \varrho}(\vec{x}, \vec{y})$ 
           does not exist. \ Proposition  \ref{proposition  fuenf oder sechs} is proven.
   \end{proof}  
                                                            
   Now we introduce the concept of a  { \it `convex corner'}. The word  `convex' seems to be superfluous 
   in normed spaces. But later we  define also something that we shall call { \it `concave corner'}.   
   These can occur in  $ \mathsf{BW}$ spaces  which have a non-convex unit ball. This justifies the adjective `convex'.    
  \begin{definition}  \rm   \label{definition convex corner}      
          \quad Let  the pair \ $ (X , \| \cdot \| ) $  \rm  be a  $ \mathsf{BW}$ space, \ 
          let  \  $ \widehat{y} \in X $.  \   
          The vector $\widehat{y}$ \  is  called  a  { \it convex  corner }  if and only if there is another vector \ 
          $ \overline{x} \in X $ \ and there are two real numbers \  $m_- < m_+ $ \
          such that we have a pair of unit vectors for each \ $ \delta \in [ 0, 1 ] $, we have 
     \begin{align}       
        \| \delta \cdot \overline{x}+ ( 1 + \delta \cdot m_- ) \cdot  \widehat{y} \| \ = \ 1 
        \ = \ \| -\delta \cdot \overline{x} + ( 1 - \delta \cdot m_+ ) \cdot \widehat{y} \| \ .    
     \end{align}      
        \hfill  $\Box$           
  \end{definition}      
    \begin{remark}  \rm
          A convex corner is only the mathematical description of something what everybody already 
          has in his mind.
          We can imagine it as an   intersection of two straight lines of unit vectors which meet with an
          Euclidean angle of less than 180 degrees.  
                                 
       Note that from the definition follows  \ $\|\widehat{y}\| = 1 $ \ and that \ 
          $ \{ \widehat{y}, \overline{x}\} $ \ is linear independent.  Further note that 
          a space with a convex corner is not strictly curved.
    \end{remark} 
    As examples we can take the  {\it H\"older weights } $\| \cdot \|_{1} $ and 
           $ \| \cdot \|_{\infty} $  on   $\mathbbm{R}^{2}$. 
           Both spaces have four convex corners, e.g.  \ $ \left(\mathbbm{R}^{2}, \| \cdot \|_{1}\right)$ has one at 
           $(0|1)$, and \ $\left( \mathbbm{R}^{2}, \| \cdot \|_{\infty}\right)$ has one at $(1|1)$. 
           They are just the corners of  the corresponding unit spheres, i.e. the corners of the squares. 
    \begin{proposition}   \label{proposition irgendwelche}
      Let  $ ( X ,  \| \cdot \| )$ \ be a positive definite balancedly weighted vector space
            which has a convex corner. \ Let  $\Upsilon( X ,  \| \cdot \| ) = ( \nu , \mu ) $. \
            We claim  \ $ \nu = - 1 $. 
    \end{proposition} 
    \begin{proof} \quad
        We assume in the proposition a convex  corner $ \widehat{y} \in X $ and  another element 
            $\overline{x} \in X$ and two real numbers \ $  m_- < m_+ $ \ with the properties of Definition 
            \ref{definition convex corner}.  
            We get with Proposition  \ref{corollary eins} the inequality  $ \nu \leq - 1 $.  
            Let us fix a number  $ \varrho > 1$,  hence  $ -\varrho  < -1 $.  We want to find two vectors \
            $ \widetilde{v}, \widetilde{w} \in X $ \ with \ 
            $ |< \widetilde{v} | \widetilde{w} >_{\bf  -\varrho}| \ > 1 $. \   This would mean that
            the angle \  $ \angle_{\bf -\varrho}(\widetilde{v}, \widetilde{w})$   does not exist.

   We define for each  $ \delta \in [ 0, 1 ] $ the pair of unit vectors  $ \vec{v}, \vec{w} $,  
    $$  \vec{v} :=  \delta \cdot \overline{x} + ( 1 + \delta \cdot m_- ) \cdot  \widehat{y} \ \quad \text{and} \ \quad 
        \vec{w} := -\delta \cdot \overline{x} + ( 1 - \delta \cdot m_+ ) \cdot \widehat{y}  \ .   $$ 
    We use the abbreviations  
          ${ \bf \Delta} =  { \bf \Delta}( \vec{v}, \vec{w})  = { \bf s}^{2} - { \bf d}^{2}  $ \ and \ 
          ${ \bf \Sigma} =  { \bf \Sigma}( \vec{v}, \vec{w})  = { \bf s}^{2} + { \bf d}^{2}  $
          from the beginning of the section `An Infinite Set of Angles', and  we compute     
         \begin{align}   \ \  < \vec{v} \: | \:  \vec{w} >_{\bf  -\varrho } \  \ = & \  \ 
                 <  \delta \cdot \overline{x}+ ( 1 + \delta \cdot m_- ) \cdot  \widehat{y} \ \: | 
                 \:  -\delta \cdot \overline{x} + ( 1 - \delta \cdot m_+ ) \cdot \widehat{y} >_{\bf  -\varrho }   \\
             =  & \ \ \|\vec{v}\| \cdot \|\vec{w}\| \cdot \frac{1}{4} \cdot  { \bf \Delta} \cdot
                  \left(  \frac{1}{4} \cdot  { \bf \Sigma}   \right)^{\bf -\varrho} \\
             =  & \ \ 1 \cdot 1 \cdot
                 \frac{\frac{1}{4} \cdot {\bf \Delta}}{\left( \frac{1}{4} \cdot {\bf \Sigma} \right)^{\bf \varrho}} 
                 \ = \    \frac{ \frac{1}{4} \cdot \left({ \bf s}^{2} - { \bf d}^{2}\right)}
                 { \left[ \frac{1}{4} \cdot \left( {\bf s}^{2} + {\bf d}^{2}\right)\right]^{{\bf \varrho}} }  \\    
             =  &  \ \frac{ \frac{1}{4} \cdot 
                 \left(  \| \left( 2 + \delta \cdot (m_- - m_+)\right) \cdot  \widehat{y} \|^{2} -  
                 \|  2 \cdot \delta \cdot \overline{x} + \delta \cdot (m_- + m_+) \cdot  \widehat{y} \|^{2} \right)}
                 { \left[ \frac{1}{4} \cdot 
                 \left(  \| \left( 2 + \delta \cdot (m_- - m_+)\right) \cdot  \widehat{y} \|^{2} +  
                 \|  2 \cdot \delta \cdot \overline{x} + \delta \cdot (m_- + m_+) \cdot  \widehat{y} \|^{2} 
                 \right) \right]^{\bf \varrho} }  \\  
             = & \ \frac{  \frac{1}{4} \cdot \left( \left( 2 + \delta \cdot (m_- - m_+)\right)^{2} \cdot  
                 \|\widehat{y}\|^{2} \ - \ \delta^{2} \cdot 
                 \| 2 \cdot \overline{x} + (m_- + m_+) \cdot  \widehat{y} \|^{2}  \right) }
                 { \left[ \frac{1}{4} \cdot \left( \left( 2 + \delta \cdot (m_- - m_+)\right)^{2} \cdot  
                 \|\widehat{y}\|^{2} \ + \ \delta^{2} \cdot 
                 \| 2 \cdot \overline{x} + (m_- + m_+) \cdot \widehat{y} \|^{2} \right) \right]^{\bf \varrho}} \ . \\   
                 \text{Hence} \  \label{equality fuenfzehn}
                  &  \  \  < \vec{v} \: | \:  \vec{w} >_{\bf  -\varrho } \ 
              = \  \frac{\frac{1}{4} \cdot {\bf \Delta}}{\left[ \frac{1}{4} \cdot 
                                                         {\bf \Sigma} \right]^{\bf \varrho}} \ \
              =  \ \frac{ 1 + \delta \cdot (m_- - m_+) + \frac{1}{4} \cdot \delta^{2} \cdot  \mathsf{K}_- }
                 { \left[ 1 + \delta \cdot (m_- - m_+) + \frac{1}{4} \cdot \delta^{2} \cdot  \mathsf{K}_+ 
                 \right]^{\bf \varrho} } \ , 
        \end{align} 
        if we define two real constants $  \mathsf{K}_- \ , \mathsf{K}_+ $ by setting 
          $$ \mathsf{K}_- := (m_- - m_+)^{2} - \|2 \cdot \overline{x} + (m_- + m_+) \cdot \widehat{y} \|^{2} \ , \quad
             \mathsf{K}_+ := (m_- - m_+)^{2} + \|2 \cdot \overline{x} + (m_- + m_+) \cdot \widehat{y} \|^{2} \ . $$  
      The above chain of identities holds for all   $ \delta \in [0, 1]$. \
       For  a shorter display we abbreviate the parts  of the fraction by  
      \begin{align*}   \label{gleichung fuenfzehneinhalb}
            \mathsf{T} \ := \ &  \frac{1}{4} \cdot  { \bf \Delta} \ = \
                            1 + \delta \cdot (m_- - m_+) + \frac{1}{4} \cdot \delta^{2} \cdot  \mathsf{K}_-   \ , \\
            \mathsf{B} \ := \ &  \frac{1}{4} \cdot  { \bf \Sigma} \ = \
                            1 + \delta \cdot (m_- - m_+) + \frac{1}{4} \cdot \delta^{2} \cdot  \mathsf{K}_+ \ .
      \end{align*}
      Since \  $  \mathsf{K}_- < \mathsf{K}_+ $ \  and \ $ m_- - m_+ < 0 $ \
            we can find a positive number $ \mathsf{s}$ with $0 < \mathsf{s} < 1$ such that we have 
            for all positive $ \delta$ with $0 < \delta \leq  \mathsf{s} $ \ the inequality
        \begin{align}      
                 0 <  \mathsf{T} <  \mathsf{B} < 1 , \ \ \text{ i.e. } \  
                  1 \ < \  \frac{ \log(\mathsf{T})} { \log(\mathsf{B})} \ . 
       \end{align}           
       
           Our aim is to find vectors $ \vec{v}, \vec{w}$ such that the product 
            $ < \vec{v} | \vec{w} >_{\bf  -\varrho}$ \ is greater than $1$. 
            With Equation  (\ref{equality fuenfzehn}) \ this is equivalent to  
        \begin{align*} 
               < \vec{v} \: | \:  \vec{w} >_{\bf  -\varrho } \  
               \ = \ \frac{\mathsf{T} } {[\mathsf{B}]^{\bf \varrho }} \ > \ 1  \ \
               \Longleftrightarrow \ \  &  \ \log(\mathsf{T})\ >  \ {\bf  \varrho} \cdot \log(\mathsf{B})  \\
               \Longleftrightarrow \ \  &  \ \frac{ \log(\mathsf{T})}{ \log(\mathsf{B})} \ < \ { \bf  \varrho} \ ,
               \quad \text{note that } \log(\mathsf{B}) \text{ is negative, see }(\ref{gleichung fuenfzehneinhalb}) \ .
        \end{align*}   
        By the rules of L'Hospital we get the limit
        \begin{align*} 
                 \lim_{\delta \searrow 0} \ \ \left(  \frac{ \log(\mathsf{T})}{ \log(\mathsf{B})} \right) \ = \ 1
        \end{align*}   
        With   Inequation (\ref{gleichung fuenfzehneinhalb}) it follows that we can find a very small \
        $ \widetilde{\delta}$ \ with \ $ 0 < \widetilde{\delta} < \mathsf{s} $ \ such that 
    $$    1 \ < \  \frac{ \log(\mathsf{T})} { \log(\mathsf{B})} \ < \ {\bf  \varrho}  $$
        is fulfilled.   That means with the definition of \
         $$  \widetilde{v} := \widetilde{\delta} \cdot \overline{x} + 
                  ( 1 + \widetilde{\delta} \cdot m_- ) \cdot  \widehat{y} \quad  \text{ and } \quad
             \widetilde{w} := -\widetilde{\delta} \cdot \overline{x} + 
                  ( 1 - \widetilde{\delta} \cdot m_+ ) \cdot \widehat{y}$$
         we get the desired inequality  \   $ < \widetilde{v} | \widetilde{w} >_{\bf  -\varrho} \ > 1 $.    
         Hence  the  ${\bf  -\varrho}$-angle \ 
            $ \angle_{-\varrho}(\widetilde{v}, \widetilde{w})$   does not exist.  Since the variable  \
            $ { \bf -\varrho} $ \ is an arbitrary number less than $-1$,  \
             Proposition \ref{proposition irgendwelche} is proven. 
    \end{proof} 
    \begin{corollary}
        Let $ ( X ,  \| \cdot \| ) \in  \mathsf{NORM} $. Further we assume that $( X, \| \cdot \| )$ has a convex
        corner. \\
        \centerline{  It follows \ $\Upsilon( X ,  \| \cdot \| ) = ( -1 , 1 ) $.   }
    \end{corollary} 
    \begin{proof}  This is a direct  consequence of  Proposition  \ref{proposition  fuenf oder sechs}
                   and  Proposition  \ref{proposition irgendwelche}.
    \end{proof}   
    \begin{corollary}    It holds for the   {\it H\"older weights }  
           $\| \cdot \|_{1}$ and  $\| \cdot \|_{\infty}$ on $\mathbbm{R}^{2}$    \\
           \centerline{   $\Upsilon(\mathbbm{R}^{2}, \| \cdot \|_{1}) \ = \ 
           \Upsilon(\mathbbm{R}^{2}, \| \cdot \|_{\infty}) \ = \ ( -1 , 1 ) $.   }
    \end{corollary} 
         Now we introduce a corresponding definition of `concave corners'. Note that they can not occur in normed
         spaces. In a normed space the triangle inequality $\widehat{(3)}$ holds, as a consequence its unit
         ball is  `everywhere' convex.      
    \begin{definition}   \rm     \label{definition concave corner}        
          \quad Let  the pair \ $ ( X , \| \cdot \| \   ) $  \ \rm  be a  \ $ \mathsf{BW}$ space, \
          let  \  $ \widehat{y} \in X $.   \   
          $ \widehat{y} $ \  is  called  a  { \it concave   corner } \  if and only if \                      
          there is  an \ $ \overline{x} \in X $,  \
                   and there are two real numbers \  $m_- < m_+ $ \
          such that we have a pair of unit vectors for each \ $ \delta \in [ 0, 1 ] $, we have           
        \begin{align}       
                 \| \delta \cdot \overline{x}+ ( 1 + \delta \cdot m_+ ) \cdot  \widehat{y} \| \ =  \ 1 
                 \ = \ \| -\delta \cdot \overline{x} + ( 1 - \delta \cdot m_- ) \cdot \widehat{y} \| \ . 
        \end{align}        $ { } $  \hfill  $\Box$                        
    \end{definition}                 
    \begin{remark}  \rm
          Note that from the definition follows \ $\|\widehat{y}\| = 1 $ \ and that \ 
          $ \{ \widehat{y}, \overline{x}\} $ \ is linear independent. 
          Further note that a space $( X , \| \cdot \| )$ with a concave corner contains a piece of a straight line
          which is completely in its unit sphere, i.e. $ ( X ,  \| \cdot \| )$ \ is not strictly curved.
    \end{remark}
    We get a set of balanced weights on \ $\mathbbm{R}^{2}$ \ if we define  for every  \ $ r \geq 0 $ \  a weight  \ \ 
           $ \| \cdot \|_{{ \tt hexagon},r} :   \mathbbm{R}^{2} \longrightarrow \mathbbm{R}^{+} \cup \{0\} $, 
           if  we fix the unit sphere  $ {\bf S}$ \ of $ ( \mathbbm{R}^{2}, \| \cdot \|_{{ \tt hexagon},r} ) $ 
           with the polygon    through the six  points \ $ \{ (0|1), (1|r), (1|-r), (0|-1),(-1|-r),(-1|r) \}$ \
           and returning to $(0|1)$,  and then extending \ $\| \cdot \|_{{ \tt hexagon},r}$ \ by  homogeneity. \ \ 
           (See Figure 1).   
  \newpage
  %
     
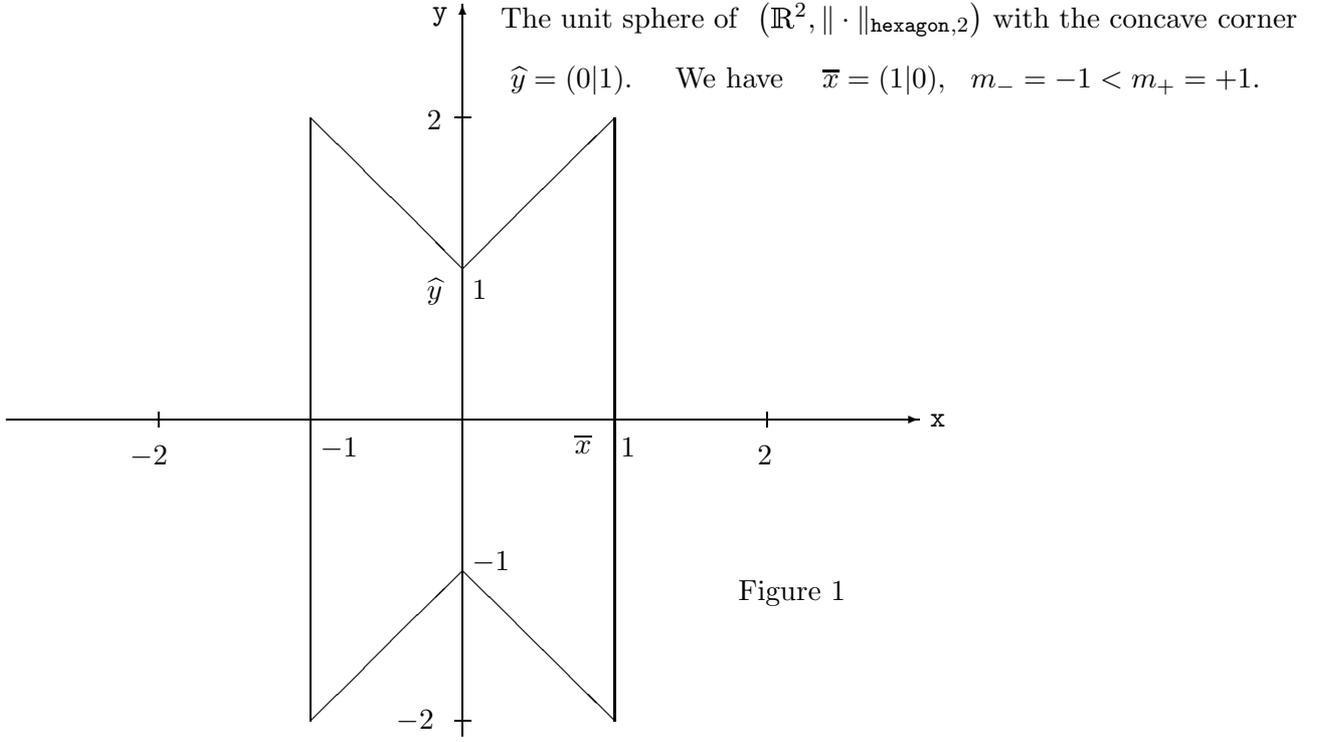
\begin{figure}[ht]    \centering
          \setlength{\unitlength}{10mm}
          \begin{picture}(6,5)     \thinlines  
                  \put(6.15,-0,1){$   {\tt x}   $}   \put(-0.4,5.3){$  {\tt y} $}    
                  \put(-6.0,0){\vector(1,0){12.0}}   \put(0,-4.2){\vector(0,1){9.7}}  
                  \put(-2.0,-4.0){\line(0,1){8.0}}   \put(2.0,4.0){\line(0,-1){8.0}} 
                  \put(-2.0,4.0){\line(1,-1){2.0}}   \put(0.0,2.0){\line(1,1){2.0}}   \put(0.0,-2.0){ $ -1 $} 
                  \put(0.0,1.6){ $ 1 $ }   \put(-0.6,1.6){ $ \widehat{y} $ } 
                  \put(1.35,-0.45){ $ \overline{x} $ }   
                  \put(-2.0,-4.0){\line(1,1){2.0}}   \put(0.0,-2.0){\line(1,-1){2.0}} 
                       \put(-0.6,3.85){ $2$}    \put(-1.0,-4.10){ $-2$} 
                  \put(-0.1,4.0){\line(1,0){0.2}}     \put(-0.1,-4.0){\line(1,0){0.2}}   
                    \put(1.95,-0.5){ $ 1 $}  \put(-2.0,-0.5){ $ -1 $} 
                    \put(3.75,-0.6){ $2$}  \put(-4.5,-0.6){ $-2$} 
                   \put(-4.0,-0.1){\line(0,1){0.2}}     \put(4.0,-0.1){\line(0,1){0.2}}     
                  \put(0.5,5.2){The unit sphere of \ $\left( \mathbbm{R}^{2}, \| \cdot \|_{{ \tt hexagon},2}
                                  \right) $  with the concave corner }  
                  \put(0.5,4.4){   $ \widehat{y} = (0|1)$.  \quad 
                                We have \quad $ \overline{x} = (1|0), \ \ m_- = -1 < m_+ = +1 $.  } 
                   \put(3.5,-2.4){ Figure 1 }              
         \end{picture}
     \end{figure}	    
    $ $      \\  \\  \\  \\  \\  \\  \\  \\ \\  \\ \\ 
     Note that the balanced weights \ $ \| \cdot \|_{{ \tt hexagon},0} $ \ and \ $ \| \cdot \|_{{ \tt hexagon},1} $ 
                \ on \ $ \mathbbm{R}^{2}$ \ coincide with the H\"older weights
                $ \| \cdot \|_{1} $ and  $ \| \cdot \|_{\infty} $, respectively, which have been defined in the 
                third section.   Further, the pairs $ \left( \mathbbm{R}^{2} ,  \| \cdot \|_{{ \tt hexagon},r} \right)$
                are normed spaces if and only if $ 0 \leq r \leq 1 $. 
     \begin{lemma}     \it      
           For all \  $ r > 1 $  \ the space \ \ $ \left( \mathbbm{R}^{2} ,  \| \cdot \|_{{ \tt hexagon},r}\right) $ \
           has  a  concave corner at \ $ \widehat{y} := (0|1) $, \ with \ 
           $\overline{x} := (1|0), \ m_- := 1 - r < 0 < m_+ := r - 1 $.   
     \end{lemma} 
     \begin{proof}  Follow Definition  \ref{definition concave corner} of a concave corner.    
     \end{proof}   
     \begin{proposition}  \quad   \it    Here we consider the special angle  $ \angle_{\bf 0} $. \\
             Let  the pair \ $ ( X , \| \cdot \| \   ) $  \ \rm  be a  \ $ \mathsf{BW}$ space, \   
             \it   let  \  $ \widehat{y} \in X $  \ \ be a  concave   corner.   \  Then  the triple
             \ $ ( X , \| \cdot \|, < . \: | \: . >_{\bf 0} ) $  \  does  not fulfil the   {\sf CSB } inequality,
             i.e. \  $ ( X, \| \cdot \| )  \notin \mathsf{pdBW}_{\bf 0} $. 
      \end{proposition} 
     \begin{proof}  \quad   We use the vectors \  $ \widehat{y}, \ \overline{x} $ \ from the above definition 
          of a concave corner,  and  then  for all \ $ \delta \in [ 0, 1 ] $ \ we take the unit vectors \ 
          $ \vec{v} := \delta \cdot \overline{x}+ ( 1 + \delta \cdot m_+ ) \cdot  \widehat{y}$ \ and \
          $ \vec{w} := -\delta \cdot \overline{x} + ( 1 - \delta \cdot m_- ) \cdot \widehat{y}$, \ and  we compute 
          \begin{align}
                 < \vec{v} \: | \:  \vec{w} >_{\bf 0} \ &  =  \ 
                 <  \delta \cdot \overline{x}+ ( 1 + \delta \cdot m_+ ) \cdot  \widehat{y} \ \: | 
                 \:  -\delta \cdot \overline{x} + ( 1 - \delta \cdot m_- ) \cdot \widehat{y} >_{\bf 0}   \\
                 &  =   \     \frac{1}{4} \cdot 1 \cdot 1 \cdot 
                 \left[  \ \left\| \ [ 2 + \delta \cdot ( m_+ - m_- )] \cdot \widehat{y} \ \right\|^{2} \  - 
                 \ \left\| \   2 \cdot    \delta \cdot \overline{x} + \delta \cdot ( m_+ + m_- ) \cdot 
                 \widehat{y} \ \right\|^{2} \  \right]   \\
                 & =   \    \frac{1}{4} \cdot   
                 \left[  \  [ 2 + \delta \cdot ( m_+ - m_- )]^{2} \cdot  \left\| \ \widehat{y} \ \right\|^{2} \  - 
                 \ \delta^{2} \cdot \left\| \   2 \cdot \overline{x} +  ( m_+ + m_- ) \cdot 
                 \widehat{y} \ \right\|^{2} \  \right]   \\
                 & =  \   1 + \delta \cdot ( m_+ - m_- ) + \frac{1}{4} \cdot \delta^{2} \cdot 
                 \left[ \ ( m_+ - m_- )^{2} \ -  \   \left\| \ 2 \cdot \overline{x} +  ( m_+ + m_- ) \cdot 
                 \widehat{y} \ \right\|^{2} \ \right]        \\  
                 & = \   1 + \delta \cdot ( m_+ - m_- ) + \frac{1}{4} \cdot \delta^{2} \cdot \mathsf{K} \, , \quad  
                 \text{ if we define the real constant} \ \ \mathsf{K} \ \ \text {by }
         \end{align}           
         \centerline{ $ \mathsf{K} \ :=  \ ( m_+ - m_- )^{2} \ -  \ \left\| \ 2 \cdot \overline{x} + 
                 ( m_+ + m_- ) \cdot \widehat{y} \ \right\|^{2} $.  }     \\  \\
                  This calculation holds for all \ $ \delta \in   [ 0, 1 ] $.  \  Hence, because of \ 
                  $  m_+ - m_- > 0 $, \  there is a positive but very small \ $ \widetilde{\delta}$ \ such that 
                  for the two unit vectors   
          $$  \widetilde{v} := \widetilde{\delta} \cdot \overline{x}+ ( 1 +  \widetilde{\delta} \cdot m_+ ) \cdot  
              \widehat{y}   \quad \ \text{and}  \quad \ \widetilde{w} := - \widetilde{\delta} \cdot \overline{x} + 
              ( 1 -  \widetilde{\delta} \cdot m_- ) \cdot \widehat{y} \  $$   
           we get \ $ < \widetilde{v} | \widetilde{w} >_{\bf 0} \ > 1$, \ i.e. the { \sf CSB } inequality is
           not satisfied and the angle $ \angle _{\bf 0} (\widetilde{v}, \widetilde{w})$ does not exist. 
           It follows \  $ ( X, \| \cdot \| )  \notin \mathsf{pdBW}_{\bf 0} $. 
      \end{proof}
      \begin{corollary}   \it     
             For all \  $ r > 1 $  \ the space \  
             $ \left( \mathbbm{R}^{2} ,  \| \cdot \|_{{ \tt hexagon},r} , < . \: | \: . >_{\bf 0} \right) $ \ \
             does not fulfil the \  { \sf CSB } inequality. \ Hence, there are vectors \ 
             $ \vec{v} \neq \vec{0} \neq \vec{w}$ \ such that  the   angle $ \angle_{\bf 0} (\vec{v}, \vec{w}) $ \ 
             is not defined.  Hence, for \ $ r > 1 $ \ it means 
             \ $( \mathbbm{R}^{2}, \| \cdot \|_{{ \tt hexagon},r} ) \notin \mathsf{pdBW}_{\bf 0} $.  
      \end{corollary} 
     \begin{proposition}   \label{proposition zehn} 
          Let \ $ \alpha, \beta $ \ be two real numbers with \ $ \alpha < -1 < \beta $. \ We have the proper inclusions
          \\ \centerline{ $\mathsf{pdBW}_{\bf \alpha} \subset \mathsf{pdBW}  \supset \mathsf{pdBW}_{\bf \beta} $ . }        \end{proposition}
      \begin{proof} From Proposition \ref{proposition irgendwelche} we know  \
           $\mathsf{pdBW}_{\bf \alpha} \neq  \mathsf{pdBW}$. For instance, the H\"older norm  
           $\| \cdot \|_{\infty}$ on $\mathbbm{R}^{2}$ has convex corners, hence it follows \ 
           $ (\mathbbm{R}^{2}, \| \cdot \|_{\infty}) \notin \mathsf{pdBW}_{\bf \alpha} $, but 
           $ (\mathbbm{R}^{2}, \| \cdot \|_{\infty}) \in \mathsf{pdBW}_{-1}$.    
                                                                        
           Now we consider $ -1 < \beta $.  Let us take the spaces $( \mathbbm{R}^{2}, \| \cdot \|_{{ \tt hexagon},r})$
           which are defined above. The balanced weight $\| \cdot \|_{{ \tt hexagon},r}$ is not a norm if $ r > 1$,
           since it has a concave corner at $ (0|1)$.  We take the unit vectors $\vec{v} := (1|r)$  and 
           $\vec{w} := (-1|r)$.  We compute $ < \vec{v} \: | \: \vec{w} >_{\bf -\varrho}$  \ for an arbitrary 
           positive number $ {\bf \varrho}$, \ i.e. $ -{\bf \varrho} < 0 $, \ and we get 
           \begin{align}
                 < \vec{v} \: | \:  \vec{w} >_{\bf -\varrho}  \ \  
                  =  & \ \ \|\vec{v}\| \cdot \|\vec{w}\| \cdot \frac{1}{4} \cdot  { \bf \Delta} \cdot
                  \left(  \frac{1}{4} \cdot  { \bf \Sigma}   \right)^{\bf -\varrho} \\
                  =  & \ \ 1 \cdot 1 \cdot
                  \frac{1}{4} \cdot \left( { \bf s}^{2} - { \bf d}^{2} \right)
                  \cdot \left( \frac{1}{4} \cdot \left( {\bf s}^{2} + {\bf d}^{2} \right) \right)^{\bf -\varrho} \\                       =  & \  \ \frac{1}{4} \cdot \left[ (2 \cdot r)^{2} - 2^{2} \right] \cdot 
                      \left( \frac{1}{4} \cdot \left( (2 \cdot r)^{2} + 2^{2} \right) \right)^{\bf -\varrho} \\
                  =  & \ \ ( r^{2} - 1 ) \cdot ( r^{2} + 1 )^{\bf -\varrho} \
                  = \ \ \frac{r^{2} - 1 } {( r^{2} + 1 )^{\bf \varrho} } \ .   
           \end{align}  
           We assume the inequality \ $ < \vec{v} \: | \:  \vec{w} >_{\bf -\varrho} \ \ > 1 $. This is equivalent to
           $ ( r^{2} - 1 ) > ( r^{2} + 1 )^{\bf \varrho} $, \ and also to 
       \begin{align}    \label{ungleichung fuenfzehn} 
           \frac{ \log( r^{2} - 1 )} {\log( r^{2} + 1 ) } > {\bf \varrho}  \ .  
       \end{align}          
           By the rules of L'Hospital we can  calculate the limit of the last term, and we get  
   $$   \lim_{ r \rightarrow \infty} \ \ \frac{ \log( r^{2} - 1 )} {\log( r^{2} + 1 ) }  \ = \ 1 \ . $$ 
           Hence for all \  $ 0 < {\bf \varrho} < 1$, i.e.  $ -1 < -{\bf \varrho} < 0$, \ we find a big number 
           $ R $  such that  Inequality (\ref{ungleichung fuenfzehn}) is fulfilled with $ r := R $. 
           This means \ $ < \vec{v} \: | \:  \vec{w} >_{\bf -\varrho} \ \ > 1 $, \
           i.e. that the angle $ \angle_{\bf -\varrho}(\vec{v}, \vec{w})$ does not exist in \
            $( \mathbbm{R}^{2}, \| \cdot \|_{{ \tt hexagon},R})$. \ We get that the space
           $( \mathbbm{R}^{2}, \| \cdot \|_{{ \tt hexagon},R})$ is not an element of $\mathsf{pdBW}_{\bf -\varrho}$.
           Since \  $( \mathbbm{R}^{2}, \| \cdot \|_{{ \tt hexagon},R})  \in \mathsf{pdBW}_{\bf -1} $ \
           we get that Proposition  \ref{proposition zehn} is proven.  
           
      \end{proof}
      \begin{proposition}   \label{proposition elf} 
          Let \ $ \alpha, \beta $ \ be two positive real numbers with \
          $ -\alpha < -1 < 1 < \beta $. \ We have the proper inclusions \\
          \centerline{ $\mathsf{NORM}_{\bf -\alpha} \subset \mathsf{NORM}  \supset \mathsf{NORM}_{\bf \beta} $ . }    
      \end{proposition} 
      \begin{proof}
             This follows directly from Proposition \ref{proposition  fuenf oder sechs} and
             Proposition  \ref{proposition irgendwelche}.
             For instance, the H\"older norm  $\| \cdot \|_{\infty}$ on $\mathbbm{R}^{2}$ has convex corners, hence it
             is not strictly convex. That means that the space $\left( \mathbbm{R}^{2}, \| \cdot \|_{\infty} \right)$
             neither is an element of $ \mathsf{NORM}_{\bf -\alpha} $, nor an element of  $ \mathsf{NORM}_{\beta}$. 
      \end{proof} 
    \section{Some Conjectures}
    	We formulate two open questions.  
    	\begin{conjecture} \label{drittes theorem}
        	Let us take four positive real  numbers \ $\alpha, \beta, \gamma , \delta $ \	with \\ 
	                \centerline{   $ -\delta < -\gamma < -1 <  1 < \alpha <  \beta $.}  \\
        	From Corollary \ref{zweites Korollar}  we know  \
	   $$ \mathsf{NORM}_{\bf -\delta} \subset   \mathsf{NORM}_{\bf -\gamma} \subset    \mathsf{NORM}
	                \supset  \mathsf{NORM}_{\bf \alpha} \supset  \mathsf{NORM}_{\bf \beta}    \  ,     $$  
	        and from Proposition  \ref{proposition elf}  we have   
	        $   \mathsf{NORM}_{\bf -\gamma} \neq  \mathsf{NORM}  \neq  \mathsf{NORM}_{\bf \alpha} $.   \      
	        We are convinced that in fact all four inclusions are proper, and we believe that all five
	        classes are different.                                      
    	\end{conjecture}
	    \begin{conjecture} \label{viertes theorem}
	          Let us assume four real  numbers \ $\alpha, \beta, \gamma , \delta$,
	          with \\ 
	          \centerline{   $ -\delta < -\gamma < -1 <  \alpha <  \beta $.}   \\
            We already know from Proposition  \ref{proposition zehn} the inequalities \
            $ \mathsf{pdBW}_{\bf -\gamma} \neq  \mathsf{pdBW}  \neq  \mathsf{pdBW}_{\bf \alpha}  $. 
            We believe that in fact we have four proper inclusions        
	     $$   \mathsf{pdBW}_{\bf -\delta} \subset   \mathsf{pdBW}_{\bf -\gamma} \subset    \mathsf{pdBW}
	          \supset  \mathsf{pdBW}_{\bf \alpha} \supset  \mathsf{pdBW}_{\bf \beta} \ , $$
	          and that all five classes are different, too.             
      \end{conjecture}     
      In Section  \ref{section three} \ `Some Examples of Balancedly Weighted Vector Spaces'  we defined  
      the set of balanced weights  $\| \cdot \|_{p}$ on $ \mathbbm{R}^{2}$,  the { \it H\"older weights }. 
      We use here only positive $ p $, for $\vec{x} = (x , y ) \in  \mathbbm{R}^{2} $  we set 
      $\|\vec{x}\|_{p} := \sqrt[p]{|x|^{p}+|y|^{p}}$.  
      For such  $ p > 0$ the weight $ \| \cdot \|_{p} $ is positive definite.  
      The pairs $ (\mathbbm{R}^{2}, \| \cdot \|_{p}) $ may be a supply of suitable examples to prove or disprove the 
      above conjectures.     
      $ $   \\   \\  
      Now we we say something about finite products of  $ \mathsf{BW}$ spaces, and we ask interesting questions. 
      We just have mentioned the H\"older weights, defined in third section.
      The method we used there can be generalized to construct products. Note that  we restrict our description 
      to products  with have only two factors. But this can be extended to a finite number of factors very easily.  
                                    
   Assume two real vector spaces $ A, B $ provided with a balanced weight, i.e. we have two $ \mathsf{BW}$ spaces
      $ \left(A , \| \cdot \|_A \right), (B , \| \cdot \|_B) $, both spaces are not necessarily positive definite.
      Let $p$ be any element from  the extended real numbers, i.e. $ p \in \mathbbm{R} \cup \{ -\infty , +\infty \}$. 
      If $ A \times B $ denotes the usual  cartesian product of the vector spaces $A$ and $B$, we define a 
      balanced weight $ \| \cdot \|_{p} $ for   $ A \times B $. If \ $p$ \ is a positive real number
      we define (corresponding to the definition in the third section)
      for an element  $ \left(\vec{a}, \vec{b} \right) \in A \times B$ \ the real numbers
   \begin{align}   
     \left\|\left(\vec{a}, \vec{b}\right) \right\|_{p}   \ := \ &
     \sqrt[p]{\|\vec{a} \|_{A}^{p} +  \|\vec{b} \|_{B}^{p}} \qquad \text{for the positive number { \it p}, \ and} \\ 
     \left\|\left(\vec{a}, \vec{b}\right) \right\|_{-p} \ := \ &            
             \begin{cases}
             \sqrt[-p]{ \|\vec{a}\|_A^{-p} + \|\vec{b}\|_B^{-p}}   &   
             \quad \mbox{if} \quad  \left\|\vec{a} \right\|_A \cdot  \| \vec{b} \|_B  \neq 0   \\    
             \ \ 0 & \quad \mbox{if} \quad  \left\|\vec{a} \right\|_A \cdot \| \vec{b} \|_B = 0 \ .   \\                                                                 \end{cases}      
     \end{align}              
        To make the definition complete we set \
        $  \left\|\left(\vec{a}, \vec{b}\right) \right\|_{0}  := 0 $, \ and   
      $$ \left\|\left(\vec{a}, \vec{b}\right) \right\|_{\infty}  
                    :=  \max \left\{  \left\| \vec{a} \right\|_{A} , \| \vec{b} \|_{B} \right\} \ , \quad
         \left\|\left(\vec{a}, \vec{b}\right) \right\|_{-\infty}  
                    :=  \min \left\{  \left\| \vec{a} \right\|_{A} ,  \| \vec{b} \|_{B} \right\} \ .    $$    
                                                  
    It is easy to verify some properties of $ \| \cdot \|_{p} $. For instance, the weight 
        $ \| \cdot \|_{p} $ is positive definite if and only if $ p > 0$ and  both $\| \cdot \|_A$ and  
        $\| \cdot \|_B$ are positive definite. Further, the pair \ $ ((A \times B) ,  \| \cdot \|_{p}) $ \
        is a   normed space if and only if $ p \geq 1$ and  both $\| \cdot \|_A$ and  $\| \cdot \|_B$ are norms. 
        Further,  the pair \ $ ((A \times B) ,  \| \cdot \|_{p}) $ \ is an inner product space if and only
        if \ $ p = 2 $ \ and both $\| \cdot \|_A$ and  $\| \cdot \|_B$ \ are inner products.   \\
                                
   The next conjecture deals with a more intricate problem.   
   \begin{conjecture}                                                  
        We take four real vector spaces provided with a positive definite balanced weight, i.e. we have \
        $ (A , \| \cdot \|_A), (B , \| \cdot \|_B), (C , \| \cdot \|_C), (D , \| \cdot \|_D) \in  \mathsf{pdBW}$. 
        \  Let us assume the identities   
        $$ \Upsilon \left(A , \| \cdot \|_A \right) =   \Upsilon( C , \| \cdot \|_C ) \quad \text{and}  \quad
           \Upsilon( B , \| \cdot \|_B ) =   \Upsilon( D , \| \cdot \|_D ) . $$ 
        Then we conjecture that    
     $$ \Upsilon( (A \times B) ,  \| \cdot \|_{p} ) \ = \  
           \Upsilon( (C \times D) ,  \| \cdot \|_{p} ) \quad \text{holds for an arbitrary} \ \ p > 0 .  $$                       \end{conjecture}
  $ $ \\ 
   At the end we try to find an `algebraic structure' on the class  $\mathsf{pdBW}_{\bf \alpha} $,
   for a fixed  number $ \alpha $. For two elements  
   $ (A , \| \cdot \|_A), (B , \| \cdot \|_B) $ of $ \mathsf{pdBW}_{\bf \alpha}$ \
   we look  for  a  weight  $ \| \cdot \|_{A \times B} $  on $ A \times B $ such that the pair
   $ \left( A \times B,\| \cdot \|_{A \times B} \right) $  is an element of  $ \mathsf{pdBW}_{\bf \alpha} $, too. 
   Before we make the conjecture we consider an example. 
                                       
 Take two copies of the  real numbers   
   $ \mathbbm{R} $  provided with the usual Euclidean metric $ | \cdot | $.   The pair  
   $ ( \mathbbm{R},  | \cdot |) $ is an inner product space and hence an element of  
   $ \mathsf{pdBW}_{\bf \varrho}$  for all real numbers $ \varrho $, see  Theorem  \ref{zweites theorem}.   
   We take the Cartesian product   $ \mathbbm{R}^{2} := \mathbbm{R} \times  \mathbbm{R} $  and we provide it with a 
   H\"older weight  $\| \cdot \|_{p}$. But the pair   $ (\mathbbm{R}^{2}, \| \cdot \|_{p}) $ is an inner product space
   only for $ p= 2$.  Hence it is an element of the classes $ \mathsf{pdBW}_{\bf \varrho}$ for each $ {\bf \varrho} $ 
   only for $ p=2$. \  This example leads to a natural question. 
    \begin{conjecture}   
             Let $ \alpha $  be a  fixed real number. Let \  
             $ (A , \| \cdot \|_A), (B , \| \cdot \|_B) $ be two elements of $ \mathsf{pdBW}_{\bf \alpha}$,
             i.e. they have the angle  $ \angle_{\bf \alpha}$. \ We consider the product vector space $ A \times B $.   
       $$   \text { We ask whether the positive definite $ \mathsf{BW}$ space } \
            \left(A \times B, \| \cdot \|_2\right)  \text{ has the angle } \angle_{\bf \alpha}  \text{, too} .  $$  
     \end{conjecture}                                               
      { $ $ }   \\ 
		  { \bf Acknowledgements: }   \rm  We like to thank Prof. Dr. Eberhard  Oeljeklaus 
     who supported us by interested discussions and important advices, and who suggested many improvements. \  
     Further we thank Berkan G\"urgec for a lot of technical aid.
  \bibliographystyle{mn}
     { $ $ }   \\ 
 %
 %
   \end{document}